\documentclass{amsart}
\usepackage{a4wide}

\usepackage{amssymb}
\usepackage{amsmath}
\usepackage{amsthm}
\usepackage{stmaryrd}
\usepackage[utf8]{inputenc}
\usepackage{graphicx}

\usepackage{bm,bbm}
\usepackage{ifthen}
\usepackage{algorithm}
\usepackage{algpseudocode}

\usepackage{hyperref}
\usepackage{booktabs}

\usepackage{tikz}
\usepackage{pgfplots}
\pgfplotsset{compat=1.18}
\usepackage{tikz-3dplot}

\newtheorem{remark}{Remark}
\newtheorem{theorem}{Theorem}

\numberwithin{equation}{section}

\newcommand{\vek}[1]{\bm{#1}}
\newcommand{\ten}[1]{\mathbf{#1}}
\newcommand{ \T}{\mbox{\scriptsize $\mathsf{t}$}}
\newcommand{\bomega}{\boldsymbol{\omega}}
\newcommand{\bPhi}{\boldsymbol{\Phi}}
\newcommand{\bvphi}{\boldsymbol{\varphi}}
\newcommand{\bsigma}{\boldsymbol{\sigma}}

\newcommand{\Grad}{{\bf Grad\,}}
\newcommand{\bepsilon}{\boldsymbol{\epsilon}}
\newcommand{\bDiv}{{\bf Div\,}}
\newcommand{\curl}{{\bf curl\,}}

\newcommand{\Tr}[1][]{ \mathrm{tr} \ifthenelse{\equal{#1}{}}{}{#1}}
\renewcommand{\div}[1][]{\mathrm{div} \ifthenelse{\equal{#1}{}}{}{#1}}
\renewcommand{\curl}[1][]{\mathrm{curl} \ifthenelse{\equal{#1}{}}{}{#1}}
\newcommand{\dev}[1][]{ \mathrm{dev}\ifthenelse{\equal{#1}{}}{}{\left\{ {#1} \right\}}  }

\newcommand{\W}{\mathcal{W}}
\renewcommand{\L}{\mathcal{L}}

\newcommand{\R}{\mathbbm{R}}
\newcommand{\I}{\mathbbm{I}}
\newcommand{\F}{\mathcal{F}}
\newcommand{\TT}{\mathcal{T}}
\newcommand{\RT}{\mathcal{R}\mathcal{T}}

\newcommand{\vL}{\vek{L}}
\newcommand{\tL}{\ten{L}}
\newcommand{\vH}{\vek{H}}
\newcommand{\tH}{\ten{H}}

\newcommand{\VHone}{ \vH^1(\Omega, \R^d) }
\newcommand{\vHmone}{ \vH^{-1}(\Omega, \R^d) }
\newcommand{\Ltwo}{ L^2(\Omega, \R) }
\newcommand{\VLtwo}{ \vL^2(\Omega, \R^d) }
\newcommand{\TLtwo}{ \tL^2(\Omega, \R^{d\times d}) }
\newcommand{\Linfty}{ L^{\infty}(\Omega, \R) }
\newcommand{\HDiv}{ \vH(\div;\Omega, \R^d) }

\newcommand{\HCurlDiv}{ \tH(\curl\,\div;\Omega, \R^{d\times d}) }

\newcommand{\vX}{\vek{X}}
\newcommand{\vx}{\vek{x}}

\newcommand{\tI}{\ten{I}}

\newcommand{\vN}{\vek{N}}
\newcommand{\vn}{\vek{n}}

\newcommand{\vT}{\vek{T}}
\newcommand{\vt}{\vek{t}}

\newcommand{\tQ}{\ten{Q}}

\newcommand{\vB}{\vek{B}}

\newcommand{\vP}{\vek{P}}

\newcommand{\tG}{\ten{G}}

\newcommand{\vU}{\vek{u}}
\newcommand{\sU}{u}
\newcommand{\tP}{\ten{P}}
\newcommand{\tF}{\ten{F}}
\newcommand{\vUh}{\tilde{\vek{u}}}
\newcommand{\vUUh}{\vek{U}}

\newcommand{\vR}{\vek{R}}

\newcommand{\tA}{\ten{A}}
\newcommand{\tB}{\ten{B}}
\newcommand{\tC}{\ten{C}}

\newcommand{\tK}{\ten{K}}

\newcommand{\vzero}{\vek{0}}
\newcommand{\tzero}{\ten{0}}

\newcommand{\Pol}{\ensuremath{\mathcal{P}}}
\newcommand{\cof}{\ensuremath{\mathrm{cof}\,}}

\newcommand{\TG}{\bar{\vT}}
\newcommand{\TGn}{\bar{T}_n}
\newcommand{\BG}{\bar{\vB}}
\newcommand{\UG}{\bar{\vU}}
\newcommand{\UGn}{\bar{\sU}_n}

\title{A four-field mixed formulation for incompressible finite elasticity}

\author{Guosheng Fu}
\address{Guosheng Fu, 
Department of Applied and Computational Mathematics and Statistics, University of Notre Dame,
Notre Dame, IN 46556, USA}
\email{gfu@nd.edu}

\author{Michael Neunteufel}
\address{Michael Neunteufel, Fariborz Maseeh Department of Mathematics and Statistics, Portland State University, 1825 SW Broadway, 97201 Portland, Oregon, USA}
\email{mneunteu@pdx.edu}

\author{Joachim Schöberl}
\address{Joachim Schöberl, Institute of Analysis and Scientific Computing, TU Wien,
Wiedner Hauptstr. 8-10, A-1040 Wien, Austria}
\email{joachim.schoeberl@tuwien.ac.at}

\author{Adam Zdunek}
\address{Adam Zdunek, HB Berrit, Solhemsbackarna 73, SE-163 56 Sp\aa nga, Sweden}
\email{adam.zdunek@berrit.se}


\begin{document}

\begin{abstract}
 In this work, we generalize the mass-conserving mixed stress (MCS)
 finite element method for Stokes equations \cite{Gopalakrishnan2019}, involving normal velocity and tangential-normal stress continuous fields, to incompressible finite elasticity. By means of the three-field Hu--Washizu principle, introducing the displacement gradient and 1st~Piola--Kirchhoff stress tensor as additional fields, we circumvent the inversion of the constitutive law. We lift the arising distributional derivatives of the displacement gradient
 to a regular auxiliary displacement gradient field. Static condensation can be applied at the element level, providing a global pure displacement problem to be solved. We present a stabilization motivated by Hybrid Discontinuous Galerkin methods. A solving algorithm is discussed, which asserts the solvability of the arising linearized subproblems for problems with physically positive eigenvalues. The excellent performance of the proposed method is corroborated by several numerical experiments.
 \\
	\vspace*{0.25cm}
	
\end{abstract}

\maketitle

\section{Introduction}
\label{sec:Introduction}
This paper is concerned with the construction of a numerical robust mixed finite element method in incompressible finite elasticity, extending the pressure-robust mass-conserving mixed stress (MCS) method for Stokes equations \cite{Gopalakrishnan2019,Lederer2019,Gopalakrishnan2020}. 
The Stokes equations governing stationary incompressible fluid mechanics are 
analogous to the Navier--Lam\'{e} equations governing the incompressible linear elasticity. Needless to say, these two fields
span a huge and important part of classic mechanics. Their importance in computational mechanics as equivalent 
 mixed displacement-pressure variational boundary value formulations in the Galerkin--Bubnov sense is widely
recognized. It is also well-understood that if divergence-free test functions were available in general, we could
compute the weakly divergence-free displacement separately. The incompressibility of the displacement fields would guarantee the exact preservation of mass in linear Stokes fluid mechanics and Navier--Lam\'e elasticity. In the absence of ideal choices for the displacement and pressure test spaces, workaround numerical mixed formulations of saddle point problems are forced to carefully match the used spaces to fulfill the celebrated \mbox{\emph{inf-sup}} stability condition. The situation in finite elasticity is more complicated due to the nonlinear form of the volume constraint in finite elasticity.
%constraint of the determinant of the deformation gradient to be unity. 
Several methods have been developed for small and large strain elasticity. In the latter, the nonlinear equations are solved by applying a Newton-Raphson algorithm, linearizing the equation around a configuration. In comparison to linear elasticity, these linearized problems involve terms from the previous configuration making rigorous theoretical a-priori convergence analysis in general unfeasible.

A three-field Hu--Washizu formulation based on an isochoric splitting has been proposed in \cite{SimoTaylorPister85} and recently improved by \cite{schonherrRobustHybridMixed2022} for linear and quadratic hexahedral and tetrahedral meshes. Both of them make use of the Flory split \cite{Flory61} of the deformation gradient. We call them \emph{STP} methods (after the authors Simo, Taylor, and Pister of \cite{SimoTaylorPister85}). Enhanced strain methods \cite{SR90,KT00a,RS95} split the strain additively into a compatible part, associated with the displacement fields, and an enhanced part. Extensions to the nonlinear case have been proposed in \cite{SA92,KT00b}. A two-field Hellinger--Reissner assumed strain element was presented in \cite{VSW19} circumventing the inversion of the nonlinear material law by an iterative computation procedure. Mixed methods based on polyconvex strain energy-functionals were proposed \cite{PB20,BGO15,BGO16,SWB11}, where the deformation gradient, cofactor matrix, and determinant are treated separately.

A locking-free brick element based on reduced integration and stabilization has been proposed in \cite{RWR98}, and the equivalence to the nonlinear enhanced strain methods and B-bar method has been discussed in \cite{R02}. Further, equivalence with Discontinuous Galerkin (DG) methods by choosing a specific matrix-valued stabilization factor was shown \cite{RBW17}. A low-order hybrid Discontinuous Galerkin (HDG) formulation for large deformations on quadrilateral meshes has been derived in \cite{WBAR17}. An adaptive stabilization based on the local value of the elastic moduli within elements for DG in nonlinear elasticity was proposed in \cite{TECL2008}. The choice of stabilization parameters for HDG methods in nonlinear elasticity and thin structures has been discussed in \cite{TNBP2019}.

Based on a Hu--Washizu formulation and Hilbert complexes for nonlinear elasticity \cite{AY2016} compatible strain mixed finite element methods (\emph{CSMFEM}) have been proposed for 2D compressible \cite{ASY17}, incompressible \cite{Shojaei2018}, and 3D compressible and incompressible \cite{Shojaei2019} elasticity. A mixed method for nonlinear elasticity based on finite element exterior calculus was presented in \cite{dhasMixedMethod3D2022}.

(H)DG methods \cite{cockburnNewHybridizationTechniques2005} provide more flexibility compared to methods based on continuous Lagrange displacement fields, preventing several locking problems. However, the involved stability parameter needs to be chosen sophistically to obtain a good balance between stability and accuracy. Mixed methods also entail flexibility, and the hybridized degrees of freedom (DoFs) are often equivalent to HDG methods. However, no stability parameter is needed in the linear case so that mixed methods have the potential to be more robust. For the incompressible Stokes equations, the mass-conserving mixed stress (MCS) formulation proposed in \cite{Gopalakrishnan2019,Lederer2019,Gopalakrishnan2020} entails the same advantages of the H(div)-conforming HDG method \cite{LS15} such as exactly divergence-free solutions and pressure-robustness, but without the necessity of a stabilization parameter. We expect that expanding an exactly divergence-free, pressure-robust, and mass-conserving method will lead to a robust method for nonlinear incompressible elasticity. Therefore, the MCS method is a guideline for this work.

By the already mentioned analogy between
the Stokes equations and the Navier--Lam\'{e} equations, one readily has the corresponding
normal displacement tangential-normal stress continuous  (\emph{NDTNS}) MCS method transferred from the Stokes equations to incompressible linear elasticity. Further transfer by generalization to finite elasticity is not possible in general. For the classic linear theories, the MCS method relies on the 
inversion of the constitutive law. In other words, in the elasticity case it relies on the
compliance strain-stress form of the material law. It can be rewritten in an equivalent three-field mixed formulation including a strain field avoiding the inversion of the constitutive law at the cost of an additional unknown, cf. the discussion in Section~\ref{subsec:small_strain}. Another important characteristic of the MCS method is its relaxation of the regularity of the displacement, from standard 
$\VHone$ to normal continuous $\HDiv$
while strengthening it, from $\VLtwo$ in the classic formulation 
to $\HCurlDiv$ \cite{Lederer2019} for the stress, where henceforth $d$ denotes the dimension. This possibility
is achieved by avoiding eliminating the stress between the material law and the equilibrium 
equation as it is done in the classic displacement-pressure formulations.  
In this work we circumvent the inversion of the material law and keep the flexibility of relaxing/strengthening the regularity of the independent fields by employing the Hu--Washizu approach
\cite{HU55,Washizu1975}. The delicate matching of the spaces for stability and robustness is
addressed in some detail. The considerations are similar to those in the sister tangential-displacement normal-normal stress continuous (TDNNS) method \cite{Neunteufel2020,Neunteufel2021}.
The nonlinearity of the formulation becomes an obstacle. The classic Total Lagrangian formulation (TL) in terms of the energy conjugate 2\textsuperscript{nd} (symmetric) Piola--Kirchhoff stress and the Green-Lagrange strain and its push-forward Updated Lagrangian formulation (UL) add difficulties of multiplying distributions because of the weak regularity of displacement fields \cite[Section 3.2]{Neunteufel2020}. These are avoided using a formulation in terms of the 1\textsuperscript{st} (non-symmetric) Piola--Kirchhoff stress paired with the (non-symmetric) deformation gradient.  

The remainder of the paper is structured as follows. In the next section, we introduce the notation and the basic equations of finite elasticity and derive the Hu--Washizu formulation of our method on the continuous level. Section~\ref{sec:discrete_setting} introduces the mixed finite element method and implementation techniques such as hybridization, static condensation, and reduction of degrees of freedom are discussed. In Section~\ref{sec:solver} we discuss possible stabilization procedures and present a theoretical result that each linearization step within Newton-Raphson is solvable for the case that no negative eigenvalues are expected in the physical problem. Further, we show that the original MCS method is recovered in the small-strain regime. In Section~\ref{sec:Numerical results} we present several benchmark examples to demonstrate the performance of the proposed method. We compare the results with the \emph{CSMFEM} and \emph{STP} methods as well as a Taylor--Hood-like discretization method.

\section{Finite elasticity}
\label{sec:Finite elasticity}

Consider a domain $\Omega\subset \R^d$, $d=2,3$, associated with the reference configuration of the body of interest with a sufficiently smooth boundary $\Gamma = \partial\Omega$ and oriented unit normal $\vN$ at $\vX$ on $\Gamma$.
Let $\vU: \Omega \mapsto \R^d$ 
\[
\vU(\vX):=\bvphi(\vX)-\vX,
\]
denote the displacement vector field, where 
$\bvphi: \Omega \mapsto \R^{d}$ is the deformation mapping. Any time dependence is henceforth suppressed. Further, let 
$J(\Grad\bvphi):=\det(\Grad\bvphi) > 0$ denote the Jacobian mapping (Jacobian determinant) measuring the volume change at $\vX\in\Omega$ in the sense $dv = JdV$. The corresponding tangential mapping is $\Grad\bvphi$. It maps material line elements in the sense  $d\vx=\Grad\bvphi\,d\vX$. The injective linear transformation $\Grad\bvphi$ is also known as the deformation gradient. The relation between the deformation gradient and the displacement gradient  follows from the definition of the material displacement as
\begin{equation}
\Grad\bvphi=\tI+\Grad\vU,
\end{equation}
where $\tI$ is the identity matrix, and $\Grad\bvphi$ and $\Grad\vU$ are two-point tensors.
The material displacement vector field $\vU(\vX)$ will henceforth be a primary variable. The deformation gradient
computed from the displacements will henceforth be represented weakly in the $\TLtwo$-sense
by the auxiliary linear transformation (deformation gradient) denoted $\tF$. It will be accomplished using the like-wise 
auxiliary 1\textsuperscript{st}~Piola--Kirchhoff stress tensor\footnote{Sometimes called the engineering stress tensor since it measures the stress as force per unit reference area.}  denoted $\tP$ as a 
Lagrange multiplier. Given the normal vector $\vN$ in the reference configuration,
we compute the so-called engineering stress vector denoted $\vT$ at $\vX$ as a force per unit referential area $dA$, using the 
1\textsuperscript{st}~Piola--Kirchhoff stress as  $\tP\vN=\vT$.

Henceforth, in the continuous case, we will only allow isochoric exact deformations. Thus, in any deformation $\bvphi$ the volume will
be preserved, i.e.,  $J=1$. In the discrete setting, we will use a Lagrange multiplier to enforce the isochoric constraint.

In the following,
we consider volumetric constraints by response functions $C:\R^{+}\rightarrow\R$ such that
$C(J)=0$ if and only if $J=1$. Examples being,
\begin{equation}
C(J)=J-1 \qquad \text{and}\qquad C(J)=\ln J.
\label{eq:constraint-like-volumetric-response-function}
\end{equation}
The stress reaction to the volumetric constraints \eqref{eq:constraint-like-volumetric-response-function} is the pressure denoted $p$. It is also the energy conjugate variable
in the Hill sense to the volume ratio $J$. 
Throughout, we assume the constitutive material response to be hyperelastic determined by a scalar 
strain energy potential $\W:\R^{d\times d} \rightarrow \R$. We may consider a material non-homogeneity $\W=\W(\vX,\cdot)$, but it is suppressed for transparency. For simplicity we
confine to an isotropic and frame-indifferent (objective) material and use a neo-Hooke like strain energy per unit reference volume,
\begin{equation}
 \widetilde{\W}(\tF) = \dfrac{\mu}{2}(\tF:\tF-d)-\mu\ln\det\tF,
 \label{eq:neo-Hookean-like-strain-energy_old}
\end{equation}
where $\tA:\tB=A_{ij}B_{ij}$ denotes the Frobenius inner product between matrices and $\mu>0$ is the generalized shear modulus. By our assumptions, the energy potential depends on the right Cauchy--Green strain tensor $\tC = \tF^{\T}\tF$, $\widetilde{\W}(\tF)=\overline{\W}(\tC)$. However, as we use $\tF$ as an auxiliary field, we will write the potential in terms of $\tF$ throughout this work. In effect,
combining \eqref{eq:constraint-like-volumetric-response-function} and
\eqref{eq:neo-Hookean-like-strain-energy_old} we get the following\footnote{This formulation is not volumetric-iscochoric decoupled in the sense of \cite{SimoTaylorPister85,Zdunek1986,GAH2000}.} 1\textsuperscript{st}~Piola--Kirchhoff stress,
\begin{equation}
 \tP=-\tilde{p}C^{\prime}(J)\dfrac{\partial\det\tF}{\partial\tF}+\dfrac{\partial\widetilde{\W}}{\partial\tF}(\tF)=-\tilde{p}C^{\prime}(J)J\tF^{-\T} + \mu\left(\tF-\tF^{-\T}\right)=-\tilde{p}C^{\prime}(J)\cof\tF + \mu\left(\tF-\tF^{-\T}\right),
 \label{eq:total-stress-equation}
\end{equation}
where $\cof\tF$ denotes the cofactor matrix of $\tF$.

Using the incompressibility constraint $J=1$ in \eqref{eq:total-stress-equation} yields $C^\prime(J)=1$ in both cases for \eqref{eq:constraint-like-volumetric-response-function} and further
\begin{align}
  \tP = \mu\,\tF-(\tilde{p}+\mu)\tF^{-\T}. 
\end{align}
Thus, absorbing $\mu$ into a new pressure $p= \tilde{p}+\mu$ we can use the simpler and in incompressible elasticity commonly used (see, e.g. \cite{Shojaei2019}) neo-Hooke law
\begin{align}
  \label{eq:neo-Hookean-like-strain-energy}
  \W(\tF) = \frac{\mu}{2}(\tF:\tF-d),\qquad \mu>0
\end{align}
and thus $\tP=\frac{\partial \W}{\partial \tF}(\tF)-pC^\prime(J)\tF^{-\T}$.

\begin{remark}
  We note that for $\widetilde{\W}(\tF)$ the extra stress, the second term in \eqref{eq:total-stress-equation}, is required to 
vanish in the reference configuration, where $\tF=\tI$. The stress reaction $\tilde{p}$ will be determined 
by the equilibrium equations and the boundary conditions and is also required to vanish in the reference configuration, when all sources, fluxes, and prescribed displacements are zero.

Using $\W(\tF)$ instead, the pressure $p$ is, in general, not zero in the absence of external impressions, but $p=\mu$. However, the same displacement solutions are obtained. Further, we observed a significantly improved stability behavior for $\W(\tF)$ compared to $\widetilde{\W}(\tF)$ in the numerical experiments. A possible explanation for the worse performance of $\widetilde{\W}(\tF)$ is that the variations of\, $\ln\det \tF$ yield a more nonlinear and, therefore, complicated behavior for the Newton-Raphson method. A complete investigation is, however, out of scope of this work. Nevertheless, this observation clearly motivates the use of $\W(\tF)$ for the remainder of this work.
\end{remark}

Data to the  variational mechanical boundary value problems we consider, consist of 
%the material parameter $\mu>0$, 
the body force $\bar{\vB}=\bar{\vB}(\vX)$ per unit reference volume, 
a prescribed boundary material displacement vector field
$\UG=\UG(\vX)$ on $\Gamma_{D}$, and a boundary stress vector field $\TG=\TG(\vX)$ on $\Gamma_{N}$. We assume that $\TG$ and $\bar{\vB}$ are independent of the displacement $\vU$. For example, in the case of follower loads, one can rewrite the forces with the unknown displacement field in a straightforward manner. The boundary $\Gamma$ is decomposed in the standard fashion as,
$\Gamma=\overline{\Gamma}_{D}\cup\overline{\Gamma}_{N}$ and $\Gamma_{D}\cap\Gamma_{N}=\emptyset$ into the Dirichlet and Neumann boundary. In the following, we will treat e.g. the normal and tangential components of the prescribed displacement $\UG$ separately and thus, the boundary splitting is done twice, $\Gamma=\overline{\Gamma}_{D,n}\cup\overline{\Gamma}_{N,n}$ and $\Gamma=\overline{\Gamma}_{D,t}\cup\overline{\Gamma}_{N,t}$ in the same fashion.

We want to construct a mixed displacement-stress formulation where the normal displacement component 
$\sU_{n}:=\vN\cdot\vU$, with $\vN\cdot\vU=N_iu_i$ denoting the Euclidean inner product, is an essential boundary condition on $\Gamma_{D,n}$. In addition, the tangential-normal part of the 1\textsuperscript{st} Piola--Kirchhoff stress tensor $\tP$, i.e. $\vP_{tn}:= \tQ\tP\vN$, with $\tQ:=\tI-\vN\otimes\vN$ the projection to the tangential components, will be prescribed on $\Gamma_{N,t}$ as essential boundary conditions. The essential boundary conditions are enforced to hold point-wise, whereas the natural boundary conditions involving the tangential part of the displacement and the normal-normal component of the stress are incorporated in a weak sense.

In other words, the displacement space has to allow for the definition of a normal trace, while the stress space has to allow a tangential-normal trace. We conceive that the appropriate space for the displacement $\vU$ with homogeneous essential boundary conditions is the vector space 
$\vH_{\Gamma_{D,n}}(\div,\Omega;\R^{d}):=\{\vU\in \VLtwo\,:\, \div \vU\in \Ltwo,\, \Tr_{n}\vU=0 \text{ on }\Gamma_{D,n}\}$ satisfying homogeneous normal boundary conditions point-wise on $\Gamma_{D,n}$. If $\Gamma_{D,n}=\Gamma$ we will write $\vH_{0}(\div,\Omega;\R^{d})$. Further, 
we will see that the proper space for the 1\textsuperscript{st} Piola--Kirchhoff stress tensor $\tP$ with stress-free essential boundary conditions is the matrix-valued space 
$\vH_{\Gamma_{N,t}}(\curl\,\div, \Omega; \R^{d\times d})=\{ \tP\in\Ltwo\,:\, \curl\,\div\tP\in \vHmone,\,\Tr_{tn}\tP=0 \text{ on }\Gamma_{N,t}\}$ with vanishing tangential-normal components $\vP_{tn}$ on $\Gamma_{N,t}$. Non-homogeneous boundary conditions are handled in the usual manner.
In passing, it is noted that the 1\textsuperscript{st}~Piola--Kirchhoff stress is not symmetric. Further, the auxiliary linear transformation (deformation gradient) $\tF$ is of $\TLtwo$-class. Finally,
the pressure $p$ is of  $\Ltwo$-class.

Notably, we will construct a mixed formulation where the tangential trace of the displacements, a natural boundary condition in our formulation, is imposed 
in a weak sense while the tangential component of the stress vector, an essential Dirichlet boundary condition, is imposed point-wise. This is in contrast to classical elasticity, where the tangential displacement is also imposed point-wise.
The desired mixed variational boundary value formulation is achieved using the Hu--Washizu approach \cite{HU55,Washizu1975}, which provides the necessary flexibility.

The strong formulation of the mechanical boundary value problem for the exactly incompressible case follows. 
Given data $D=\{\BG,\TG,\UG\}$ determine $\{\vU,\tF,\tP,p\}$
such that:
\begin{subequations}
  \label{eq:NDTNS-strong-formulation}
  \begin{align}
    \tP &= -pC^{\prime}(J(\tF))\cof{\tF} + 
    \dfrac{\partial\W(\tF)}{\partial\tF}& &\text{in}\quad \Omega, \label{eq:NDTNS-strong-formulation_a}\\
    \bDiv{\tP}  &= -\BG& &\text{in}\quad \Omega,\label{eq:NDTNS-strong-formulation_b}\\
    C(J(\tF))       &= 0 & &\text{in}\quad \Omega,\label{eq:NDTNS-strong-formulation_c}\\
    \tF         &= \tI + \Grad\vU & &\text{in}\quad \Omega,\label{eq:NDTNS-strong-formulation_d}\\
    \sU_{n} :=\vN\cdot\vU          &\doteq \UGn & &\text{on}\quad \Gamma_{D,n},\label{eq:NDTNS-strong-formulation_e}\\
    \vU_{t} :=\vU-\sU_{n}\vN    &= \UG_t & &\text{on}\quad \Gamma_{D,t},\label{eq:NDTNS-strong-formulation_f}\\
    {P}_{nn} :=\vN\cdot\tP\vN  &= \TGn& &\text{on}\quad \Gamma_{N,n},\label{eq:NDTNS-strong-formulation_g}\\
    \vP_{tn}:=\tQ\tP\vN    &\doteq \TG_{t}& &\text{on}\quad \Gamma_{N,t},\label{eq:NDTNS-strong-formulation_h}
  \end{align}
\end{subequations}
%\end{array}\;\right\}
%\end{equation}
where $\doteq$ denotes the point-wise equality and the volume constraint-like function $C(\cdot)$ is given in \eqref{eq:constraint-like-volumetric-response-function}. Further, here the strain energy function $\W(\cdot)$ is specified by \eqref{eq:neo-Hookean-like-strain-energy}. 
The spaces for the independent variables and their test functions are presented in Table~\ref{tab:function-spaces}. Note that the 1\textsuperscript{st}~Piola--Kirchhoff stress 
$\tP$ will be acting as a tensor-valued Lagrange multiplier to enforce the equality
\eqref{eq:NDTNS-strong-formulation_d}. 

 \begin{table}[!ht]
 \begin{center}
  \begin{tabular}{lll}\hline\\
    total displacement, $\vU(\vX)$, & $\vH(\div)$,  & $\R^{d}$ \\
   auxiliary 1\textsuperscript{st}~Piola--Kirchhoff stress, $\tP$, & $\tH(\curl\,\div)$, & $\R^{d\times d}$\\
   auxiliary tangent map, $\tF$, &  $\tL^{2}$, &  $\R^{d\times d}$\\ 
    stress reaction, pressure, $p$, &  ${L}^{2}$, &  $\R$\\
   \hline
  \end{tabular}
   \caption{Function space setting. None of the auxiliary variables $\tF$, $\tP$, or $p$ are computed  (derived) from the referential displacement $\vU$.
   \label{tab:function-spaces}}
 \end{center}
\end{table}

Next we consider the variational boundary value formulation corresponding to  \eqref{eq:NDTNS-strong-formulation}, which reads:\\
For given data $D=\{\BG,\TG,\UG\}$ find $(\vU,\tF,\tP,p)\in \HDiv\times \TLtwo\times \HCurlDiv\times \Ltwo$ such that  $\sU_n = \UGn$, $\vP_{tn}= \TG_{t}$, and for all $(\delta\vU,\delta\tF,\delta\tP,\delta p)\in \vH_{\Gamma_{D,n}}(\div;\Omega,\R^d)\times \TLtwo\times \vH_{\Gamma_{N,t}}(\curl\,\div;\Omega,\R^{d\times d})\times \Ltwo$ there holds
\begin{subequations}
	\label{eq:NDTNS-bvp}
	\begin{alignat}{5}
	&\int_{\Omega}\Big(\frac{\partial\W}{\partial\tF}(\tF)[\delta \tF] &&-  \tP:\delta\tF &&&&-p\,C'(J(\tF))\cof \tF:\delta\tF\Big)\,dV&&= 0,\\
	& \int_{\Omega}(\tI-\tF):\delta\tP\,dV &&&&+\langle \Grad\vU,\delta\tP\rangle &&&&=L^{\ast}_{\mathrm{ext}}(\delta\tP),\label{eq:NDTNS-bvp_dP}\\
	& &&\langle \Grad\delta\vU,\tP\rangle &&&&&&=L_{\mathrm{ext}}(\delta\vU),\label{eq:NDTNS-bvp_dU}\\
	& \int_{\Omega}-\delta p\, C(J(\tF))\,dV&&&&&&&&=0,\label{eq:NDTNS-bvp_dp}
 	\end{alignat}
\end{subequations}
where
\begin{align}
L^{\ast}_{\mathrm{ext}}(\delta\tP) = \int_{\Gamma_{D,t}}\UG_t\cdot\delta\vP_{tn}\,dA\quad \text{ and }\quad
L_{\mathrm{ext}}(\delta\vU) =\int_{\Omega}\BG\cdot \delta\vU\,dV + \int_{\Gamma_{N,n}}\TGn \delta\sU_n\,dA.\label{eq:ext_forces}
\end{align}

\begin{remark}[Nearly incompressible materials]
    By adding $-\frac{1}{\kappa}p\,\delta p$ to \eqref{eq:NDTNS-strong-formulation_c} and \eqref{eq:NDTNS-bvp_dp}, where $\kappa>0$ denotes the bulk modulus, the method can directly incorporate nearly incompressible materials. However, we focus on the challenging incompressible case throughout the paper.
\end{remark}

Note that only the normal component of the stress vector $\TG$ appears in \eqref{eq:NDTNS-bvp_dU} as the tangential part is set as an essential boundary condition for the 1\textsuperscript{st}~Piola--Kirchhoff stress tensor $\tP$. On the other hand also only the normal component of the prescribed displacement $\UG$ on the Dirichlet boundary is incorporated strongly in the formulation; the tangential components are satisfied weakly due to the right-hand side of \eqref{eq:NDTNS-bvp_dP}. 

The pairing $\langle \tP,\Grad \vU\rangle$ in \eqref{eq:NDTNS-bvp} cannot be understood as an $L^2$-integral since the $\HDiv$ space only guarantees that the divergence is square integrable, but not the individual components of the gradient. As discussed and analyzed in \cite{Gopalakrishnan2019,Lederer2019} the members in the function space $\HCurlDiv$ can be paired with the gradients of members in $\HDiv$ interpreted as functionals. This means that the gradient of functions in $\HDiv$ is included in the topological dual space of $\HCurlDiv$, denoted by $\HCurlDiv^*$. Furthermore, the divergence of functions in $\HCurlDiv$ is in the dual space of $\HDiv$. Thus, we interpret the pairing as a so-called duality pairing, where a functional acts on a suitable function 
\begin{align}
    \langle \tP,\Grad \vU\rangle_{\vH(\curl\,\div)\times \vH(\curl\,\div)^*} = -\langle\bDiv{\tP},\vU\rangle_{\vH(\div)^*\times \vH(\div)},\label{eq:duality_pairing_cont}
\end{align}
which we abbreviate with $\langle \tP,\Grad \vU\rangle = -\langle\bDiv{\tP},\vU\rangle$. Strictly speaking, we required here that $\tP\in \vH_{\Gamma_{N,t}}(\curl\,\div, \Omega; \R^{d\times d})$ and $\vU\in \vH_{\Gamma_{D,n}}(\div, \Omega; \R^{d})$ with $\Gamma_{N,t}\cup \Gamma_{D,n}=\Gamma$. Then the boundary terms arising from integration by parts vanish in \eqref{eq:duality_pairing_cont}. 
However, in Section~\ref{subsec:finite_elements}, we discuss this abstract concept of duality pairings in terms of finite elements accounting for boundary conditions in detail. Therein, the pairing translates to piece-wise classical integrals involving element and boundary terms.
 
We further emphasize that, strictly speaking, taking  $\tF\in \TLtwo$, does in general not provide enough regularity to guarantee that  \eqref{eq:NDTNS-strong-formulation_a} is well-defined as $C^\prime(J(\tF))\cof\tF$ is at least quadratic in $\tF$. 
Nevertheless, as we will consider finite elements consisting of piece-wise polynomials, which are always in $\Linfty$ for bounded domains $\Omega$, the question of well-posedness due to the multiplication of functions does not arise.
 
The variational problem \eqref{eq:NDTNS-bvp} can also be written in terms of a Lagrange functional
\begin{align}
\label{eq:lagr_func_pressure}
\L(\vU,\tF,\tP,p)&=\int_{\Omega}\left(\W(\tF) - p\,C(J(\tF)) \right)\,dV - \langle \tF-\tI-\Grad \vU,\tP\rangle - W_{\mathrm{ext}}(\vU,\tP),
\end{align}
where we used that $\langle \tF,\tP\rangle=\int_{\Omega}\tF:\tP\,dV$ reduces to the usual $L^2$ inner product and that with \eqref{eq:ext_forces}
\begin{align}
W_{\mathrm{ext}}(\vU,\tP)=L^{\ast}_{\mathrm{ext}}(\tP) +
L_{\mathrm{ext}}(\vU).\label{eq:ext_work}
\end{align}

Formulation \eqref{eq:lagr_func_pressure} has the form of a saddle point problem. In the discretized case, we will use hybridization techniques enabling us to statically condense out $\tF$, $\tP$, and $p$ in \eqref{eq:lagr_func_pressure} such that only displacement-based unknowns are left. 

\begin{remark}[Lifting of distributional gradient]
    In \eqref{eq:lagr_func_pressure} the deformation gradient field $\tF$ can be interpreted as an $L^2$ Riesz representative of the distributional gradient $\Grad \vU$ for $\vU\in \vH(\div;\Omega,\R^d)$, which is enforced by the Lagrange multiplier $\tP$. As we increase the regularity from a distribution to $L^2$ we also call $\tF$ the lifting of $\Grad \vU+\tI$.
\end{remark}

\section{Discrete setting}
\label{sec:discrete_setting}
In this section we describe the corresponding finite element spaces used in formulation \eqref{eq:lagr_func_pressure}, including hybridization and the involved duality pairing. 

First, we assume a triangulation $\TT=\{T\}$ of $\Omega$ consisting of (possibly polynomial curved) triangles or tetrahedrons in 2D or 3D, respectively. The set of all facets, edges in two dimensions and faces in 3D, also called the skeleton, is denoted by $\F=\{F\}$. The set of all piece-wise polynomials up to order $k$ on the elements and facets is given by $\Pol^k(\TT)$ and $\Pol^k(\F)$, respectively. Further, we denote with $\tilde{\Pol}^k(\TT)$ the set of all homogeneous polynomials of order $k$.

\subsection{Finite elements}
\label{subsec:finite_elements}
For the displacement field $\vU$ the $\vH(\div)$-conforming, vector-valued, and normal continuous Raviart--Thomas (RT) \cite{RT77} or Brezzi--Douglas--Marini (BDM) \cite{BDM85} elements can be used. We will use RT to get a better approximation for the divergence without introducing additional coupling degrees of freedom. Further, by post-processing techniques, we can construct a displacement field with increased $\Ltwo$-convergence rate, cf. \cite{FJQ2019},
\begin{align}
  \label{eq:fespace_rt}
  \begin{split}
    &\RT^k:=\{ \vU\in \HDiv\, :\, \vU|_T \in \Pol^k(T,\R^d)+\tilde{\Pol}^k(T)\,\vX\,\forall T\in\TT
      \,\wedge\, \llbracket \vU\cdot \vN\rrbracket_F=0\,\forall F\in \F_{\mathrm{int}} \},\\
    &\RT^k_{\Gamma}:=\{ \vU\in \RT^k\, :\, \sU_{n}=0 \text{ on } \Gamma \},
  \end{split}
\end{align}
where $\vX=(X_1,\dots,X_d)$ and the jump $\llbracket\cdot\rrbracket_F$ is defined as follows. We fix for each facet $F\in\F$ a unique $\vN_F$ as facet normal, which coincides with the outer normal vector on the boundary of the domain. Without loss of generality we assume that the element normal vector $\vN_{T_1}=\vN_F$ and therefore we define
\begin{align}
\llbracket \vU\cdot\vN\rrbracket_F := \vU|_{T_1}\cdot \vN_{T_1}+\vU|_{T_2}\cdot \vN_{T_2} = (\vU|_{T_1}-\vU|_{T_2})\cdot \vN_{F}\label{eq:def_normal_jump}
\end{align}
for inner facets $F\in \F_{\mathrm{int}}$ and on boundaries simply $\llbracket \vU\cdot\vN\rrbracket_F:= \vU\cdot\vN_F$, $F\in \F_{\mathrm{out}}$. If no misunderstandings are possible, we neglect the subscript $F$ of the jump brackets for ease of notation. 

A precise construction of a basis of \eqref{eq:fespace_rt} can be found e.g. in \cite{BBF13,Zaglmayr06}. An important difference to standard nodal elements is the mapping from the reference element $\hat{T}$ to a physical one $T$ used to preserve the normal continuity: If a normal continuous $\hat{\vU}_h$ is defined on the reference simplex $\hat{T}$ and $\bPhi:\hat{T}\to T$ denotes the mapping to the physical element, then the function $\vU_h$ defined on $T$ via the Piola transformation, see e.g. \cite[Section 2.1.3]{BBF13},
\begin{align}
\vU_h\circ\bPhi:= \frac{1}{\det\tG}\tG\,\hat{\vU}_h,\qquad \tG:=\Grad_{\hat{\vx}} \bPhi
\end{align}
is again normal continuous.

For the $L^2$-conforming finite element space piece-wise polynomials of order $k$ are used
\begin{align}
Q_h^k:=\Pol^k(\TT),\label{eq:fespace_l2}
\end{align}
which can easily be constructed by using e.g. a 2D or 3D Dubiner basis, respectively, cf. e.g. \cite[Sections 5.2.3 \& 5.2.6]{Zaglmayr06}. This space is considered for the pressure field $p$. 

By adding copies of $Q_h^k$ also the auxiliary deformation gradient $\tF$ could be discretized, $\tF_h\in [Q_h^{k}]^{d\times d}$. Nevertheless, a second approach of discretizing $\tF$ will be presented, giving beneficial properties discussed below.
 
The 1\textsuperscript{st}~Piola--Kirchhoff stress tensor $\tP$ is going to be discretized in the recently proposed matrix-valued tangential-normal continuous finite element space for $\HCurlDiv$
\begin{align}
\Sigma_{h,\Gamma}^k:=\{ \tP_h\in \Pol^k(\TT,\R^{d\times d})\,:\, \llbracket\tP_{h,tn}\rrbracket_F=0 \text{ for all } F\in \F,\,  \vP_{h,tn}=0\text{ on } \Gamma\}.\label{eq:fespace_hcd}
\end{align}
An explicit basis for $\Sigma_h^k$ can be found in \cite{Gopalakrishnan2019,Lederer2019}. We emphasize that the construction is based on dyadic products of barycentric coordinates such that the tangential-normal components are zero on all except for one facet on the reference element $\hat{T}$ leading to tangential-normal continuous elements. Higher polynomial orders are achieved by a hierarchical basis spanning the appropriate polynomial space. A different approach for constructing basis functions based on polytopal templates was presented in \cite{SNHZ2024}. Further, an explicit splitting between trace-free basis functions and others with non-zero traces can easily be accomplished. After defining the basis functions on the reference element, they are mapped onto the physical element $T$ with a covariant transformation \cite{Nedelec1986,BBF13} from the left and a Piola mapping from the right guarantying the tangential-normal continuity also on the physical element
\begin{align}
\tP_h\circ\bPhi:=\frac{1}{\det\tG}\tG^{-\T}\hat{\tP}_h\tG^{\T},\qquad \tG=\Grad_{\hat{\vx}}\bPhi,\quad \bPhi:\hat{T}\to T.\label{eq:trafo_ref_el_hcd}
\end{align}
The duality pairing $\langle \bDiv{\tP_h},\vU_h\rangle$ \eqref{eq:duality_pairing_cont} can be evaluated on a triangulation $\TT$ if $\tP_h\in \Sigma_h^{k}$ and $\vU_h\in \RT^k$, see e.g. \cite{Gopalakrishnan2019},
\begin{subequations}
\label{eq:duality_pairing}
\begin{align}
  \langle \bDiv{\tP_h},\vU_h\rangle_{\TT}&:= \sum_{T\in\TT}\int_T \bDiv{\tP_h}\cdot \vU_h\,dV -\sum_{F\in\F}\llbracket P_{h,nn}\rrbracket \sU_n\,dA\label{eq:duality_pairing_a}\\
  &=-\sum_{T\in\TT}\int_T \tP_h: \Grad\vU_h\,dV +\sum_{F\in\F}\vP_{h,tn}\cdot \llbracket \vU_t\rrbracket\,dA =: -\langle \tP_h,\Grad\vU_h\rangle_{\TT}.\label{eq:duality_pairing_b}
  \end{align}
\end{subequations}

The gap between the auxiliary displacement gradient $\tF_h-\tI$ and the actual displacement gradient $\Grad\vU_h$ is paired in the form $\langle \tF_h-\tI-\Grad\vU_h,\tP_h\rangle_{\TT}$, where we again use that $\langle\tF_h,\tP_h\rangle_{\TT}=\int_{\Omega}\tF_h:\tP_h\,dV$. Let us assume an affine\footnote{I.e., each element is a non-curved triangle or tetrahedron in 2D or 3D, respectively, and the Jacobian of the affine mapping $\Phi$ from the reference to the physical element is constant.} triangulation for the moment. Then the gradient of an $\vH(\div)$-conforming finite element function transforms to the reference element by
\begin{align}
(\Grad_{\vX}\vU_h)\circ\bPhi=\frac{1}{\det\tG}\tG\,\Grad_{\hat{\vx}}\hat{\vU}_h\tG^{-1}\label{eq:trafo_grad_hdiv}
\end{align}
as the gradient transforms covariantly and the $\vH(\div)$-conforming functions with the Piola transformation. Comparing with \eqref{eq:trafo_ref_el_hcd}, we deduce that the transformations are equal besides one transposition. This motivates to discretize the discontinuous deformation gradient field $\tF$ by the space consisting of transposed functions of \eqref{eq:fespace_hcd}. Then both $\Grad \vU$ and $\tF$ are mapped in the same way. However, we require $\tF$ to be only in $\Ltwo$ without tangential-normal continuity. To this end, we break the tangential-normal continuity of $\Sigma_h^k$ 
by ``doubling'' the corresponding coupling degrees of freedom and interpreting them as independent internal ones. The resulting space is denoted by $\Sigma_h^{k,\mathrm{dc}}$ (dc=discontinuous). With 
\begin{align*}
\tF_h\in F_h^k:=(\Sigma_h^{k,\mathrm{dc}})^{\T}:=\{\tA^{\T}\,:\, \tA\in \Sigma_h^{k,\mathrm{dc}}\}
\end{align*}
the deformation gradients transform equally with the gradient of the displacement field $\vU_h$.\\

\begin{remark}
  As the identity matrix $\tI$ is in the space $F_h^0$, it is equivalent to consider the displacement gradient $\tK:=\Grad \vU$ as an auxiliary field instead of the deformation gradient $\tF=\Grad\bvphi$ \cite{Shojaei2019}.
\end{remark}

\begin{remark}[Continuity of displacement and stress field]
The normal component $\sU_n=\vU\cdot \vN$ of the displacement and the tangential-normal components $\tQ\tP\vN$ of the 1\textsuperscript{st} Piola--Kirchhoff stress tensor are continuous on the triangulation by construction. Physically, the whole displacement and the stress vector $\tP\vN$ should be continuous. The continuity of these components is enforced in a weak sense by the duality pairing \eqref{eq:duality_pairing}. The term $\sum_{F\in\F}\llbracket P_{h,nn}\rrbracket \sU_n\,dA$ in \eqref{eq:duality_pairing_a} enforces the continuity of the missing normal-normal component of $\tP$ and $\sum_{F\in\F}\vP_{h,tn}\cdot \llbracket \vU_t\rrbracket\,dA$ in \eqref{eq:duality_pairing_b} yields a weakly tangential continuous displacement field. Further, \eqref{eq:duality_pairing} shows that the boundary conditions \eqref{eq:NDTNS-strong-formulation_f} and \eqref{eq:NDTNS-strong-formulation_g} are fulfilled in a weak sense together with \eqref{eq:ext_forces}. For details, we refer to \cite{Gopalakrishnan2019,Lederer2019}.
\end{remark}

\subsection{Discrete formulation}
\label{subsec:discrete_formulation}
With all ingredients at hand, we are now in the position to formulate the discretized version of \eqref{eq:lagr_func_pressure}. For given data $D=\{\BG,\TG,\UG\}$ find $(\vU_h,\tF_h,\tP_h,p_h)\in \RT^k\times F_h^{k}\times \Sigma_h^{k}\times Q_h^{k}$ such that $\sU_{h,n} = \UGn$, $\vP_{h,tn}= \TG_{t}$, and solving the saddle-point problem
\begin{align}
\label{eq:disc_lagr_func_pressure}
\L(\vU_h,\tF_h,\tP_h,p_h)&=\int_{\Omega}\left(\W(\tF_h) - p_h\,C(J(\tF_h))\right)\,dV - \langle \tF_h-\tI-\Grad \vU_h,\tP_h\rangle_{\TT} - W_{\mathrm{ext}}(\vU_h,\tP_h).
\end{align}
The specific choice of polynomial orders for \eqref{eq:disc_lagr_func_pressure} are motivated by the Stokes equations \cite{FJQ2019,Gopalakrishnan2019}, which are obtained from \eqref{eq:disc_lagr_func_pressure} in the small-strain regime, see also Section~\ref{subsec:small_strain} below.

\subsection{Hybridization, stabilization, and static condensation}
System \eqref{eq:disc_lagr_func_pressure} is a saddle point problem such that an indefinite system after assembling would be obtained. We can, however, use hybridization techniques resulting in solely displacement-based coupling degrees of freedom (DoFs) regaining a minimization problem. The  DoFs of $\tF_h\in F_h^k$ and $p_h\in Q_h^k$ are all internal ones without coupling to other elements. Only $\vU_h\in \RT^k$ and $\tP_h\in \Sigma_h^k$ involve coupling DoFs. Thus, to enable hybridization, the tangential-normal continuity of the stress space is broken, $\tP_h\in\Sigma_h^{k,\mathrm{dc}}$, and the continuity gets reinforced weakly by means of a Lagrange multiplier $\vUh_h$ living as vector field only on the skeleton $\F$. More precisely, $\vUh_h$ is a vector-valued function with zero normal component, i.e., only its tangential components enter the equations, $\vUh_h = \vUh_{h,t}$, and lives in the finite element space 
\begin{align}
\Lambda_{h,\Gamma}^k:=\{\vUh_h\in \Pol^k(\F,\R^d)\,:\, \vUh_h\cdot\vN_F=0 \text{ for all } F\in \F, \vUh_{h,t}=0\text{ on }\Gamma\}.\label{eq:fespace_hyb}
\end{align}
One can define the tangential continuous facet element $\hat{\vUh}_h$ on the reference element $\hat{T}$. Then, by use of the covariant transformation 
\begin{align*}
\vUh_h\circ\bPhi:= \tG^{-\T}\hat{\vUh}_h
\end{align*}
the tangential continuity gets preserved, comparable to N\'ed\'elec finite elements \cite{Nedelec1986}. Note that in two dimensions $\Lambda_{h}^k$ coincides with a to N\'ed\'elec finite elements of the second kind, where the interior degrees of freedom are set to zero. In three dimensions, however, $\Lambda_{h}^k$ lives only on the facets, whereas N\'ed\'elec elements start with edge degrees of freedom. Thus, these spaces are different.

The hybridized problem then reads: For given data $D=\{\BG,\TG,\UG\}$ find $(\vU_h,\vUh_h,\tF_h,\tP_h,p_h)\in \RT^k\times\Lambda_h^{k}\times F_h^{k}\times \Sigma_h^{k,\mathrm{dc}}\times Q_h^{k}$ such that  $\sU_{h,n} = \UGn$ on $\Gamma_{D,n}$, $\vUh_{h,t}= \UG_{t}$  on $\Gamma_{D,t}$ and solving the saddle-point problem
\begin{align}
\label{eq:disc_lagr_func_hyb}
\L^{\mathrm{hyb}}(\vU_h,\vUh_h,\tF_h,\tP_h,p_h)&=\int_{\Omega}\left(\W(\tF_h) - p_h\,C(J(\tF_h)) \right)\,dV - \langle \tF_h-\tI-\Grad \vU_h,\tP_h\rangle_{\TT} \nonumber\\
& - W^{\mathrm{hyb}}_{\mathrm{ext}}(\vU_h,\vUh_h) + \sum_{T \in \TT}\int_{\partial T}\vP_{h,tn}\cdot\vUh_{h,t}\,dA,
\end{align}
where
\begin{align}
W^{\mathrm{hyb}}_{\mathrm{ext}}(\vU,\vUh):=\int_{\Omega}\BG\cdot \vU\,dV + \int_{\Gamma_{N,n}}\TGn\, \sU_n\,dA + \int_{\Gamma_{N,t}}\TG_t\cdot\vUh_{t}\,dA.\label{eq:ext_work_hyb}
\end{align}
Note that the essential and natural boundary conditions are swapped when going to the hybridized version as the tangential component of the prescribed displacement is also incorporated point-wise, whereas the stress vector is now treated completely in a weak sense, cf. \cite{BBF13}. We emphasize that the hybridized system \eqref{eq:disc_lagr_func_hyb} is equivalent to its original version \eqref{eq:disc_lagr_func_pressure}.

For completeness, we discuss in detail how the Lagrange multiplier enforces the tangential-normal continuity of $\tP_h$. With the given fixed facet normal vector $\vN_F$ from the definition of the normal jump \eqref{eq:def_normal_jump} the jump of the tangential-normal stresses $\vP_{h,tn}= (\tI-\vN\otimes \vN)\tP_h\vN$ is given by $\llbracket\tP_{h,tn}\rrbracket=(\tI-\vN_F\otimes \vN_F)\tP_h|_{T_1} \vN_{F}-(\tI-\vN_F\otimes \vN_F)\tP_h|_{T_2} \vN_{F}$, as the sign of $\tP_{h,tn}$ changes over elements. Note, that in contrast the sign of $\vUh_{h,t}=(\tI-\vN_F\otimes \vN_F)\vUh_{h}$ does not change. If the surface integrals in \eqref{eq:disc_lagr_func_hyb} are re-ordered facet by facet, there are two contributions to each internal facet with normal vectors $\vN$ of opposite direction. Thus, we observe that
\begin{align}
\sum_{T \in \TT} \int_{\partial T} \vP_{h,tn}\cdot\vUh_{h,t}\, dA &= \sum_{F \in \F_{\mathrm{int}}} \int_{F} \llbracket\vP_{h,\vt_F\vn_F}\rrbracket \cdot\vUh_{h,\vt_F}\, dA+\sum_{F \in \Gamma_{N,t}} \int_{F} \vP_{h,tn} \cdot\vUh_{h,t}\, dA.\label{eq:jump_hyb}
\end{align}
The first term on the right-hand side of \eqref{eq:jump_hyb} enforces the tangential-normal continuity of $\tP$ and the second, together with $\int_{\Gamma_{N,t}}\TG_t\cdot \vUh_{h,t}\,dA$ in \eqref{eq:ext_work_hyb}, that $\tP_h$ fulfills the Neumann boundary condition $\vP_{h,tn}=\TG_{t}$. This underlines that $\Gamma_{N,t}$ is handled as a natural boundary condition in the hybridized version. Further, \eqref{eq:jump_hyb} combined with the boundary term $\int_{F}\vP_{h,tn}\vU_{h,t}\,dA$ in \eqref{eq:duality_pairing} shows that the Lagrange multiplier $\vUh_h$ has the physical meaning of the tangential component of the displacement $\vU$ \cite{Gopalakrishnan2019}.

More details on hybridization techniques can be found, e.g. in \cite{BBF13}. Now a Schur complement can be performed in every linearization step to locally eliminate $\tF_h$, $\tP_h$, and $p_h$. Details are given in Section~\ref{sec:solver}.

To improve the robustness of the method in the large deformation regime, we suggest adding the following stabilization term well-known from hybridized discontinuous Galerkin (HDG) methods \cite{TNBP2019} 
\begin{align}
  \label{eq:stab_hdg}
  \sum_{T\in\TT}\int_{\partial T}\frac{\tau}{2}\|(\vU_h-\vUh_h)_t\|^2\,dA,
\end{align}
where $\tau>0$ is the stabilization parameter, typically taken to be of order $\mathcal{O}(1/h)$. The effects of $\tau$ on the solution are discussed in Section~\ref{subsubsec:Inflation shell}. 

\subsection{Saving internal DoFs}
\label{subsec:save_DoFs}
When looking at the lifting $\langle \tF_h-\tI-\Grad\vU_h,\tP_h\rangle_{\TT}$ we observe that the full gradient $\Grad\vU_h$ gets lifted to the more regular auxiliary deformation gradient $\tF_h$. However, the divergence of $\vU_h\in \RT^k$ is already a regular function and thus, does not need to be explicitly lifted. Therefore, and as the $\HCurlDiv$ finite element space allows the explicit splitting into trace-free and spherical matrix-valued functions \cite{Gopalakrishnan2019,Lederer2019}, we can define the full new deformation gradient as
\begin{align}
\tF_h = \tF_{\dev,h}+\frac{1}{d}\div{(\vU_h)}\tI+\tI
\end{align}
and only perform the lifting on the trace-free (deviatoric) part $\langle \tF_{\dev,h}-\dev(\Grad \vU_h),\tP_{\dev,h}\rangle_{\TT}$. Here, $\dev(\tA)=\tA-\frac{1}{d}\Tr(\tA)\tI$ denotes the trace-free part of a matrix $\tA$ with $\Tr(\tA) = A_{ii}$. Note that also for the 1\textsuperscript{st}~Piola--Kirchhoff stress tensor only the trace-free part is then needed, saving in total $2\times\dim\Pol^k(\TT)$ local DoFs. Further, due to the orthogonality of the trace-free and spherical part, we can neglect the trace-free operator in the lifting of $\dev(\Grad \vU_h)$.

We emphasize that the spherical part of the 1\textsuperscript{st}~Piola--Kirchhoff stress tensor can be recovered by the postprocessing procedure
\begin{align}
\frac{\partial(\W(\tF_h)-p_hC(J(\tF_h)))}{\partial (\frac{1}{d}\div{(\vU_h)}\tI+\tI)}= \tP_{\Tr,h},\qquad \tP_h= \tP_{\dev,h}+\tP_{\Tr,h}.
\end{align}

\subsection{Final element}

\begin{figure}[ht!]
	\centering
	\includegraphics[width=0.2\textwidth]{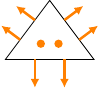}\hfill
	\includegraphics[width=0.2\textwidth]{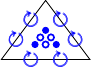}\hfill
	\includegraphics[width=0.18\textwidth]{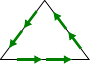}\hfill
	\includegraphics[width=0.18\textwidth]{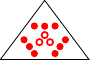}\hfill
	\includegraphics[width=0.18\textwidth]{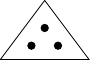}
	\caption{DoFs for lowest-order (hybridized) element $k=1$ on triangle. From left to right: $\vU\in \RT^1$, $\tP\in\Sigma_h^1$ (empty circle are spherical DoFs), $\vUh\in \Lambda_h^1$, $\tF\in F_h^1=(\Sigma_h^{1,\mathrm{dc}})^{\T}$ (empty circle are spherical DoFs), and $p\in Q_h^1$.}
	\label{fig:DoFs_element}
\end{figure}
In Figure~\ref{fig:DoFs_element} the DoFs for a single triangle for the lowest-order elements $k=1$ are visualized. For the non-hybridized element there are 35 total DoFs, whereas only 12 are coupling ones. With the techniques described in Section~\ref{subsec:save_DoFs} the number of local DoFs can further be reduced as listed in Table~\ref{tab:DoFs_combinations} for a triangle and tetrahedron for $k=1$ and $k=2$, respectively. For hybridization, all DoFs for the 1\textsuperscript{st}~Piola--Kirchhoff stress tensor $\tP\in \Sigma_h^k$ become local ones, $\tP\in \Sigma_h^{k,\mathrm{dc}}$ and the tangential facet finite element space $\Lambda_h^k$ \eqref{eq:fespace_hyb} takes the coupling DoFs from $\Sigma_h^k$. Further, by using the following relations between the number of vertices $V$, edges $E$, faces $F$, and elements $T$ in 2D and 3D for a structured mesh of a quadrilateral and cube domain, respectively,
\begin{subequations}
\label{eq:asymp_relation}
\begin{align}
        &3\#V \approx \# E,\quad 2\#V \approx \#T &\text{ in 2D},\\
&7\#V \approx \# E,\quad 12\# V \approx \# F,\quad 6\# V \approx \# T &\text{ in 3D},
\end{align}
\end{subequations}
we present the asymptotic number of DoFs on the domain with respect to the number of vertices of the mesh in Table~\ref{tab:DoFs_combinations}.
\begin{table}[!ht]
	\begin{center}
    \begin{tabular}{c|cccc}
			triangle &full &  reduced & asymptotic (reduced) \\
			\hline
			$k=1$ & 41/12 &  35/12 & 58/12\#V \\
			$k=2$ & 78/18 &  66/18 & 114/18\#V \\
		\end{tabular}\hspace*{0.3cm}
		\begin{tabular}{c|cccc}
			tetrahedron &full &  reduced & asymptotic (reduced) \\
			\hline
			$k=1$ & 115/36 &  107/36 & 534/108\#V  \\
			$k=2$ & 274/72 &  254/72 & 1308/216\#V \\
		\end{tabular}
		\caption{Total/coupling DoFs for polynomial order $k=1,2$ on a triangle (left) and tetrahedron (right) for the hybridized full or reduced method. The asymptotic number of DoFs with respect to a triangulation and the number of vertices are presented.
			\label{tab:DoFs_combinations}}
	\end{center}
\end{table}

\section{Solver}
\label{sec:solver}

In this section we will neglect the subscript $h$ for finite element functions for ease of notation. To avoid possible misunderstandings, the finite element spaces, however, will still be indicated explicitly. Further, not to complicate the presentation, we consider the full system, including spherical DoFs, as the reduced system described in Section~\ref{subsec:save_DoFs} is equivalent to the full one. 

Combining the hybridized Lagrangian \eqref{eq:disc_lagr_func_hyb} with stabilization \eqref{eq:stab_hdg} we focus on the solvability of the following Lagrangian in this section
\begin{align}
\label{eq:disc_lagr_func_hyb_stab}
\L^{\tau}(\vU,\vUh,\tF,\tP,p)&:=\L^{\mathrm{hyb}}(\vU,\vUh,\tF,\tP,p)+\sum_{T\in\TT}\int_{\partial T}\frac{\tau}{2}\|(\vU-\vUh)_t\|^2\,dA.
\end{align}

\subsection{Small strain regime}
\label{subsec:small_strain}
For completeness and to reveal the strong connection to the MCS formulation for the Stokes problem \cite{Lederer2019} we also present the boundary value problem in the small strain setting. Then, $\tP\rightarrow\bsigma$ becomes symmetric and thus, we have to use an additional Lagrange multiplier $\bomega\in [Q_h^{k-1}]^{d\times d}_{\mathrm{skew}}$ to enforce weakly the symmetry of $\bsigma\in \Sigma_h^k$ as described detailed in \cite{Gopalakrishnan2020}. Further, the nonlinear material law is assumed to degenerate to Hooke's law, represented by the fourth-order elasticity tensor $\mathbb{C}$ acting on the linearized strains. Lastly, a simple change of variables is performed using the symmetric linearized strain tensor $\bepsilon$ instead of the deformation gradient $\tF$, and we neglect the stabilization term \eqref{eq:stab_hdg} as it is not needed in the small strain regime.

The corresponding weak formulation finally reads: For given data $D=\{\BG,\TG,\UG\}$ find $(\vU,\bsigma,\bepsilon,\bomega,p)\in \RT^k\times\Sigma_h^{k}\times F_h^{k}\times[Q_h^{k-1}]^{d\times d}_{\mathrm{skew}}\times Q_h^{k}$ such that $\sU_n=\UGn$, $\vP_{tn}=\TG_t$, and for all $(\delta\vU,\delta\bsigma,\delta\bepsilon,\delta\bomega,\delta p)\in \RT^k_{\Gamma_{D,n}}\times\Sigma_{h,\Gamma_{N,t}}^{k}\times F_h^{k}\times[Q_h^{k-1}]^{d\times d}_{\mathrm{skew}}\times Q_h^{k}$
\begin{subequations}
	\label{eq:lin_prob}
\begin{alignat}{6}
&\int_{\Omega}\mathbb{C}\bepsilon:\delta\bepsilon\,dV &&+ \int_{\Omega} \bsigma:\delta\bepsilon\,dV &&&&&&-\int_{\Omega} p\,\Tr{\delta\bepsilon}\,dV&&= 0,\label{eq:lin_prob_F}\\
& \int_{\Omega}\bepsilon:\delta\bsigma\,dV &&&&&&-\langle \Grad\vU,\delta\bsigma\rangle_{\TT} &&+\int_{\Omega}\bomega:\delta\bsigma\,dV&&=L^{\ast}_{\mathrm{ext}}(\delta\bsigma),\label{eq:lin_prob_P}\\
& &&-\langle \Grad\delta\vU,\bsigma\rangle_{\TT} &&&&&&&&=L_{\mathrm{ext}}(\delta \vU),\label{eq:lin_prob_u}\\
&\int_{\Omega}\delta\bomega:\bsigma\,dV&&&&&&&&&&=0,\label{eq:lin_prob_w}\\
&\int_{\Omega}-\delta p\,\Tr{\bepsilon}\,dV&&&&&&&&&&=0,\label{eq:lin_prob_p}
\end{alignat}
\end{subequations}
with $L^{\ast}_{\mathrm{ext}}(\cdot)$ and $L_{\mathrm{ext}}(\cdot)$ given by \eqref{eq:ext_forces}. This weak formulation can be seen as a three-field approach with pressure in $(\vU,\bsigma,\bepsilon,p)$ for the MCS formulation with weakly imposed symmetry in $(\vU,\bsigma,p)$ \cite{Gopalakrishnan2020} and the unique solvability of \eqref{eq:lin_prob} can readily be proven with the therein used techniques.

\begin{theorem}[Pressure-robustness]
	\label{thm:pressure_robust}
	Let $\Gamma_{D,n}=\Gamma$ and $\BG=\Grad\Psi$ be a gradient field with the potential $\Psi$ chosen such that it has zero mean value. Let further $\UGn\equiv0$ on $\Gamma_{D,n}$, i.e., clamped or symmetry boundary conditions are prescribed. Consider the pressure space $Q^{k}_h$, and a displacement finite element space $V_h$ such that $\div{(V_h)}=Q_h^{k}$. Then ($\vU$, $\bepsilon$, $p$, $\bsigma$, $\bomega$)=($\vzero$, $\tzero$, $\Pi^k_{L^2}\Psi$, $\Pi^k_{L^2}\Psi\tI$, $\tzero$) is the unique solution of the small strain Problem~\eqref{eq:lin_prob}, where $\Pi^k_{L^2}$ denotes the $L^2$-projection into $\Pol^{k}(\TT)$.
\end{theorem}
\begin{proof}
Inserting the terms $\vU=\vzero$, $\bepsilon=\tzero$, $p=\Pi^k_{L^2}\Psi$, $\bsigma=\Pi^k_{L^2}\Psi\tI$, and $\bomega=\tzero$ in \eqref{eq:lin_prob_F}, \eqref{eq:lin_prob_P}, \eqref{eq:lin_prob_w}, and \eqref{eq:lin_prob_p} directly fulfills the equations. Next we take a closer look at \eqref{eq:lin_prob_u} and insert $\BG=\Grad\Psi$  and $\bsigma=\Pi^k_{L^2}\Psi\tI$ and use integration by parts
\begin{align}
-\int_{\Omega}\div\delta \vU\,\Pi^k_{L^2}\Psi\,dV = -\langle \Grad\delta \vU, \bsigma\rangle_{\TT} \overset{\eqref{eq:lin_prob_u}}{=} \int_{\Omega}\Grad \Psi \cdot \delta \vU\,dV = -\int_{\Omega}\Psi\, \div\delta \vU\,dV.
\end{align}
We define $\div\delta \vU =: \delta p\in Q_h^{k}$ as the finite element spaces are constructed in such a way that $\div(V_h)= Q_h^{k}$. By recognizing that 
\begin{align}
\int_{\Omega} \Pi^k_{L^2}\Psi\,\delta p\,dV=\int_{\Omega}\Psi\,\delta p\,dV,\qquad \forall \delta p\in Q_h^{k}
\end{align}
is exactly the definition of the $L^2$-projection $\Pi^k_{L^2}$, and noticing that \eqref{eq:lin_prob} is uniquely solvable, we conclude the proof for the small strain case.
\end{proof}

Theorem \ref{thm:pressure_robust} says that in the case of homogeneous Dirichlet data for the normal component of the displacement $\vU$ and if the external force $\BG$ is a gradient field, then the right-hand side gets completely compensated by the pressure $p$ and the displacement $\vU\equiv \vzero$ is exactly zero. Further, the theorem shows that the pressure and velocity space have to fulfill the crucial property $\div(V_h)= Q_h^{k}$, which is fulfilled for $V_h=\RT^k$. % (and $V_h=\BDM^{k+1}$).
In finite elasticity, to enforce the constraint $J=\det\tF=1$ exactly, i.e. point-wise, one would require a pressure of high polynomial order due to the quadratic and cubic expression in $\tF$ in 2D and 3D, respectively. This combination is, however, not stable in general and may easily lead to locking as the constraint can only be fulfilled when the displacement is zero. Therefore, only incompressibility in a weak sense can be obtained with mixed finite elements for finite elasticity available so far. This means $\int_{\Omega}C(\det(\tF))\,\delta p\,dV=0$ for all test functions $\delta p$, but not $C(\det(\tF))=0$ point-wise.

\subsection{Quasi-Newton-Raphson method}
\label{subsec:quasi_newton}
In this section we present and discuss a quasi-Newton-Raphson procedure to solve Lagrangian \eqref{eq:disc_lagr_func_hyb_stab}. We start with the first variations of \eqref{eq:disc_lagr_func_hyb_stab}, which reads with the abbreviations $\vUUh=(\vU,\vUh)$, $\delta \vUUh=(\delta \vU,\delta \vUh)$,  
\begin{subequations}
    \label{eq:NDTNS-bvp-hyb-stab}
    \begin{align}
      (\partial_{\tF}\L^{\tau})(\vUUh,\tF,\tP,p)[\delta \tF]&=\int_{\Omega}\left(\frac{\partial\W}{\partial\tF}(\tF)-p\,C'(J(\tF))\cof \tF-  \tP\right):\delta\tF\,dV,\\
       (\partial_{p}\L^{\tau})(\vUUh,\tF,\tP,p)[\delta p]&=\int_{\Omega}-\delta p\, C(J(\tF))\,dV,\\
       (\partial_{\tP}\L^{\tau})(\vUUh,\tF,\tP,p)[\delta \tP]&=\int_{\Omega}(\tI-\tF):\delta\tP\,dV+\langle \Grad\vU,\delta\tP\rangle -\sum_{T\in\TT}\int_{\partial T}\vUh_t\cdot\delta\vP_{tn}\,dA,\\
      \begin{split}
        (\partial_{\vUUh}\L^{\tau})(\vUUh,\tF,\tP,p)[\delta \vUUh]&=\langle \Grad\delta\vU,\tP\rangle-\sum_{T\in\TT}\int_{\partial T}\left(\delta\vUh_t\cdot\vP_{tn}+\tau(\vU-\vUh)_t\cdot(\delta\vU-\delta\vUh)_t\right)\,dA\\
        &\qquad-L_{\mathrm{ext}}(\delta\vU,\delta \vUh).
      \end{split}
    \end{align}
  \end{subequations}
We will use a load incremental algorithm to solve the nonlinear problem. Therefore, we need to linearize the problem at the current iteration and solve the resulting linearized problem. To this end, let $(\vUUh^n,\tF^n,\tP^n,p^n)$ be the result from the $n$-th iteration. The following iterate is then computed with the Newton-Raphson method
\begin{align}
  \label{eq:newton_raphson}
  \begin{split}
    &(\vUUh^{n+1},\tF^{n+1},\tP^{n+1},p^{n+1})=(\vUUh^{n},\tF^{n},\tP^{n},p^{n})+(\Delta\vUUh,\Delta\tF,\Delta\tP,\Delta p),\\
    &\text{with } (\Delta\vUUh,\Delta\tF,\Delta\tP,\Delta p) = -(\mathbb{K}^n)^{-1}\,\vR^n,
  \end{split}
\end{align}
where $\vR^n$ denotes the residual vector evaluated at the $n$-th iterate, $\partial\L^{\tau,n}:=(\partial\L^{\tau})( \vUUh^n,\tF^n,\tP^n,p^n)$, 
\begin{align}
  \vR^n:=\vR^n(\delta \vUUh,\delta \tF,\delta \tP, \delta p):=\begin{pmatrix}
    \partial_{\tF}\L^{\tau,n}[\delta \tF],&
    \partial_{p}\L^{\tau,n}[\delta p],&
    \partial_{\tP}\L^{\tau,n}[\delta \tP],&
    \partial_{\vUUh}\L^{\tau,n}[\delta \vUUh]
  \end{pmatrix}^{\T}
\end{align}
and $(\mathbb{K}^n)^{-1}$ is the inverse of the stiffness matrix $\mathbb{K}^n$, i.e., the linearization of \eqref{eq:NDTNS-bvp-hyb-stab} or the second variation of Lagrangian \eqref{eq:disc_lagr_func_hyb_stab} at the $n$-th iterate. The system \eqref{eq:newton_raphson} used to solve for the increment reads: Find $(\Delta \vUUh,\Delta \tF,\Delta \tP,\Delta p)\in (\RT^k\times\Lambda_h^{k})\times F_h^k\times\Sigma_h^k\times Q_h^{k}$ such that for all $(\delta\vUUh,\delta\tF,\delta\tP,\delta p)\in (\RT^k\times\Lambda_h^{k})\times F_h^k\times\Sigma_h^k\times Q_h^{k}$
\begin{subequations}
  \label{eq:linearized_prob}
  \begin{align}
    \int_{\Omega}\left((\mathbb{A}(\tF^n,p^n)\Delta \tF):\delta\tF+b(\tF^n;\Delta p,\delta\tF)-  \Delta\tP:\delta\tF\right)\,dV&= -R_1^n(\delta \tF),\label{eq:linearized_prob_F}\\
         \int_{\Omega}b(\tF^n;\delta p,\Delta\tF)\,dV&=-R_2^n(\delta p),\label{eq:linearized_prob_p}\\
         \int_{\Omega}-\Delta \tF:\delta\tP\,dV+\langle \Grad\Delta\vU,\delta\tP\rangle -\sum_{T\in\TT}\int_{\partial T}\Delta\vUh_t\cdot\delta\vP_{tn}\,dA&=-R_3^n(\delta \tP),\\
        \langle \Grad\delta\vU,\Delta\tP\rangle-\sum_{T\in\TT}\int_{\partial T}\left(\delta\vUh_t\cdot\Delta\vP_{tn}+\tau(\Delta\vU-\Delta\vUh)_t\cdot(\delta\vU-\delta\vUh)_t\right)\,dA&=-R_4^n(\delta \vUUh),\label{eq:linearized_prob_U}
  \end{align}
\end{subequations}
where $b(\tF^n;p,\tF):=-(p\, C^\prime(J^n)\,\cof\tF^n:\tF)$ and the fourth-order tensor $\mathbb{A}(\tF^n,p^n)$ is defined by
\begin{align}
\label{eq:def_stiffness_mat_law}
  \begin{split}
    (\mathbb{A}(\tF^n,p^n)\tA):\tB&:= \left(\left(\frac{\partial^2 \W}{\partial \tF^2}(\tF^n)-p^n\frac{\partial(C^\prime(J)\cof\tF)}{\partial \tF}(\tF^n)\right)\tA\right):\tB \\
    &= \mu\,\tA: \tB - p^n\left(C^{\prime\prime}(J^n)(\cof\tF^n:\tA)\cof\tF^n+\tF^n\times \tA\right):\tB.
  \end{split}
\end{align} 
Here, $(\tA\times \tB)_{ij} = \varepsilon_{ikl}\varepsilon_{jmn}A_{km}B_{ln}$ denotes the tensor cross product of $\tA$ and $\tB$ \cite{BGO16} with $\varepsilon_{ijk}$ the Levi--Civita permuting tensor. With it, there holds $\cof \tB = \frac{1}{2} \tB\times \tB$ for any tensor $\tB$ and the tensor cross product is linear in its arguments and symmetric. Thus, there holds for the variation of the cofactor matrix $\cof\tA$ in direction $\tB$ that $\partial_{\tA}\cof(\tA)[\tB]=\tA\times \tB$.

We will stabilize the pressure equation by adding a small negative mass matrix
\begin{align}
  \label{eq:pressure_reg}
  \int_{\Omega}-\varepsilon_p\, \Delta p\,\delta p\,d V,\qquad 1\gg \varepsilon_p>0. 
\end{align}
The negative sign might seem counterintuitive, but it is necessary to make the system positive definite. Adding the pressure stabilization \eqref{eq:pressure_reg} at the left-hand side of \eqref{eq:linearized_prob_p} we obtain the system
\begin{subequations}
  \label{eq:linearized_prob_stab}
  \begin{align}
    \int_{\Omega}\left((\mathbb{A}(\tF^n,p^n)\Delta \tF):\delta\tF+b(\tF^n;\Delta p,\delta\tF)-  \Delta\tP:\delta\tF\right)\,dV&= -R_1^n(\delta \tF),\label{eq:linearized_prob_stab_F}\\
         \int_{\Omega}\left(b(\tF^n;\delta p,\Delta\tF)-\varepsilon_p\, \Delta p\,\delta p\right)\,dV&=-R_2^n(\delta p),\label{eq:linearized_prob_stab_p}\\
         \int_{\Omega}-\Delta \tF:\delta\tP\,dV+\langle \Grad\Delta\vU,\delta\tP\rangle -\sum_{T\in\TT}\int_{\partial T}\Delta\vUh_t\cdot\delta\vP_{tn}\,dA&=-R_3^n(\delta \tP),\\
        \langle \Grad\delta\vU,\Delta\tP\rangle-\sum_{T\in\TT}\int_{\partial T}\left(\delta\vUh_t\cdot\Delta\vP_{tn}+\tau(\Delta\vU-\Delta\vUh)_t\cdot(\delta\vU-\delta\vUh)_t\right)\,dA&=-R_4^n(\delta \vUUh).\label{eq:linearized_prob_stab_U}
  \end{align}
\end{subequations}
We denote the resulting stiffness matrix again by $\mathbb{K}^n$. We emphasize that pressure stabilization \eqref{eq:pressure_reg} is not used for the residuum $\vR^n$. Applying static condensation to the linearized system \eqref{eq:linearized_prob_stab}, by using a Schur complement, we can eliminate all internal DoFs, especially the deformation gradient $\tF$, the 1\textsuperscript{st}~Piola--Kirchhoff stress tensor $\tP$ and the pressure $p$. The resulting linear system involving only the displacement $\vU$ and the hybridization field $\vUh$ is then a minimization problem, allowing for the use of efficient direct and iterative solvers. 

\subsection{Eigenvalue stabilization and solvability}
In practice, we solve problem
\eqref{eq:linearized_prob_stab} with the elasticity tensor
$\mathbb{A}$. However, due to the high nonlinearity of the problem, the fourth-order tensor $\mathbb{A}$ can have negative eigenvalues, although the physical problem prevents such eigenvalues. This may cause stability issues such that Newton-Raphson may converge slowly or diverge. If negative eigenvalues are expected to appear by physics, e.g., for bifurcation or buckling problems, techniques such as arc-length methods should be used. Henceforth, we assume that no physical negative eigenvalues appear. To improve solvability, we can adapt $\mathbb{A}$ as presented in the following. As we will use the changes only for the stiffness matrix and not to modify the residuum, we arrive at a quasi-Newton-Raphson method for \eqref{eq:newton_raphson}. First, we compute the smallest eigenvalue of $\mathbb{A}$
\begin{align}
  \label{eq:lambda_min}
  \lambda_{\min}(\mathbb{A}(\tF^n,p^n)):=\max\left\{\varepsilon_\lambda,-\min \lambda_{\mathbb{A}(\tF^n,p^n)}\right\},\qquad \varepsilon_\lambda\ge 0
\end{align} 
and use it to shift the tensor $\mathbb{A}$ to make it positive (semi-)definite, depending on if the constant $\varepsilon_\lambda$ in \eqref{eq:lambda_min} is chosen to be zero or positive \cite{TECL2008}. The shifted tensor reads
\begin{align}
  \label{eq:shifted_mat_tensor}
  \tilde{\mathbb{A}}(\tF^n,p^n):= \mathbb{A}(\tF^n,p^n)+\lambda_{\min}(\mathbb{A}(\tF^n,p^n))\I,
\end{align}
with $\I$ denoting the fourth-order identity tensor. 

Thanks to stabilizations \eqref{eq:shifted_mat_tensor} and \eqref{eq:pressure_reg}, the linearized subproblems can be shown to be solvable.
\begin{theorem}
  Let $(\vUUh^n,\tF^n,\tP^n,p^n)$ be the $n$-th iterate of the quasi-Newton-Raphson method. 
  Assuming   $\varepsilon_{\lambda}>0$ in \eqref{eq:lambda_min}, and
    $J^n=\det \tF^n >0$.  
    Then, the stiffness matrix $\tilde{\mathbb{K}}^n$, where $\tilde{\mathbb{A}}$ instead of $\mathbb{A}$ is used, is invertible.
\end{theorem}
\begin{proof}
  We prove that for zero right-hand side $(\Delta \vUUh,\Delta \tF,\Delta \tP,\Delta p)=(\vzero,\tzero,\tzero,0)$ is the unique solution of \eqref{eq:linearized_prob_stab} (replacing $\mathbb{A}$ by $\tilde{\mathbb{A}}$). By setting $(\delta\vUUh,\delta\tF,\delta\tP,\delta p)=(\Delta \vUUh,\Delta \tF,\Delta \tP,\Delta p)$ and using the three last equations for the first one we obtain
  \begin{align*}
    \int_{\Omega}\left((\tilde{\mathbb{A}}(\tF^n,p^n)\Delta \tF):\Delta \tF+\varepsilon_p\,(\Delta p)^2\right)\,dV+\sum_{T\in\TT}\int_{\partial T}\tau(\Delta\vU-\Delta\vUh)_t^2\,dA=0.
  \end{align*} 
  By the positive definiteness of $\tilde{\mathbb{A}}(\tF^n,p^n)$ and the non-negativity of the other terms, the only possibility to fulfill the equation is $\Delta \tF=\tzero$ and $\Delta p = 0$. Now, using that the finite element spaces for $\tP$ and $\tF$ span the same polynomial space  
    $\Pol^k(\Omega,\R^{d\times d})$   
    we directly obtain from the first equation \eqref{eq:linearized_prob_stab_F} that $\Delta \tP=\tzero$. Further, by the compatibility of the finite element spaces for $(\vU,\vUh)$ and $\tP$ from the MCS method, we obtain from the last equation \eqref{eq:linearized_prob_stab_U} that $\Delta\vU=\Delta\vUh=\vzero$, which finishes the proof.
\end{proof}

\subsection{Final algorithm}
In addition to the stabilization strategies described above, one can apply a damping algorithm to the quasi-Newton-Raphson method to reduce possible overshooting at the beginning. For completeness, we define a damping parameter $\beta\in (0,1]$, which gets multiplied with the iteration number until $1$ is reached, such that quadratic Newton convergence is obtained after the first steps
  \begin{align}
    \label{eq:damped_newton}
    \begin{split}
      &(\vUUh^{n+1},\tF^{n+1},\tP^{n+1},p^{n+1})=(\vUUh^{n},\tF^{n},\tP^{n},p^{n})+\min\{\beta\times n,1\}\,(\Delta\vUUh,\Delta\tF,\Delta\tP,\Delta p),\\
      &\text{with } (\Delta\vUUh,\Delta\tF,\Delta\tP,\Delta p) = -(\mathbb{K}^n)^{-1}\,\vR^n.
    \end{split}
  \end{align}
Further, the load stepping gets reduced by half if the Newton-Raphson method does not converge in a prescribed number of iterations or diverges. We check for each load step if the volume determinant $J_0^{n+1}:=\Pi^0_{L^2}(J^{n+1})$ interpolated into element-wise constants is positive. If not, the result is nonphysical, and we reduce the load stepping. A pseudo-code of the adaptive load stepping quasi-Newton-Raphson algorithm can be found in Algorithm~\ref{alg:newton}.

\begin{algorithm}[H] 
	\label{alg:newton_alorithm}
	\begin{algorithmic}[1]
		\State{\bf Input:} Stabilization $\varepsilon_p$. Newton-damping $\beta$. Initial, minimal load-increment $\Delta \xi_{\mathrm{init}}$, $\mathrm{tol}_{\mathrm{inc}}$. Maximal Newton iterations $n_{\mathrm{max}}$.
		\State {\bf Output:} Solution at final configuration
    \State {\bf Initialize:} $\xi=0$, $\Delta\xi = \Delta\xi_{\mathrm{init}}$, $n_{\mathrm{old}}=0$, $\mathrm{sol}^n=0$
		\While{ $\xi < 1$}
    \State $\xi := \min\{\xi+\Delta\xi,1\}$
    \State $(\mathrm{sol}_h,n_{\mathrm{it}}):= \text{Quasi-Newton-Raphson}(\mathrm{sol}^n,\xi,\beta,n_{\mathrm{max}})$ with \eqref{eq:damped_newton}
    \State Compute minimal volume determinant $J_0^n$
		\If{$n_{\mathrm{it}}<n_{\mathrm{max}}$ \textbf{and}  $J_0^n>0$}  
    \State \textit{Quasi-Newton-Raphson converged}
		\State $\mathrm{sol}^n:= \mathrm{sol}_h$
    \If{$n_{\mathrm{it}}<8$ \textbf{and} $n_{\mathrm{old}}< 8$}
    \State $\Delta\xi:=\max\{1.5\,\Delta\xi,\Delta\xi_{\mathrm{init}}\}$
    \EndIf
    \If{$n_{\mathrm{it}}>20$ \textbf{and} $n_{\mathrm{old}}> 20$}
    \State $\Delta\xi:=0.8\,\Delta\xi$
    \EndIf
		\State $n_{\mathrm{old}}:=n_{\mathrm{it}}$
		\Else
    \State \textit{Quasi-Newton-Raphson did not converged}
    \State $\xi := \xi-\Delta\xi$
		\State $\Delta \xi := 0.5\,\Delta\xi$
    \If{$\Delta\xi<\mathrm{tol}_{\mathrm{inc}}$}
    \State \textit{Problem cannot be solved}
    \State \textbf{break}
    \EndIf
		\EndIf
		\EndWhile\\
    \Return $\mathrm{sol}^n$
		\caption{Adaptive load stepping quasi-Newton-Raphson algorithm}
		\label{alg:newton}
	\end{algorithmic}
\end{algorithm}

\section{Numerical results}
\label{sec:Numerical results}
In all numerical benchmarks the neo-Hookean material law \eqref{eq:neo-Hookean-like-strain-energy} $\W(\tF):= \dfrac{\mu}{2}\left(\tF:\tF-d\right)$ is used together with the constraint equation given by \eqref{eq:constraint-like-volumetric-response-function}. We observed that $C(J)= J-1$, although not polyconvex, performed more stable than $C(J)=\ln J$.

We denote our proposed mixed finite element method in the following by \emph{NDTNS}.

\subsection{Methods}
\label{subsec:Methods}
For a comprehensive comparison with other methods, we consider the following three approaches.

\subsubsection{Standard method (conformal Stokes like)}
\label{subsubsec:Standard method}
A simple and straightforward approach for incompressible elasticity is to consider standard Lagrangian elements
\begin{align}
U_h^k:=\{ \vU_h\in \Pol^k(\TT)\,:\, \vU_h \text{ is globally continuous}  \},\label{eq:fespace_h1}
\end{align}
for the (high-order) construction  we refer to \cite{Zaglmayr06,ZT20}, and use its vector-valued version for the displacement field $\vU\in [U_h^k]^d$. For $k>1$, the pressure space can be discretized in a Taylor--Hood fashion \cite{HT73} with $p\in U_h^{k-1}$ of one polynomial degree less than the displacement. In the lowest-order case $k=1$, the displacement field is enriched with an internal cubic or quartic bubble function $B^{d+1}(T)$ in 2D or 3D, respectively, and the pressure is in $U_h^1$, i.e., a MINI element like discretization \cite{KHPA20}.

All other quantities are expressed in terms of the displacement field unknown, e.g., $\tF(\vU)= \tI+\Grad\vU$ leading to: Find $(\vU,p)\in [U^k_h]^d\times U^{k-1}_h$ (respectively $[(U^1_h\oplus \prod_{T\in\TT}B^{d+1}(T))]^d\times U^1_h$ for $k=1$) for the Lagrangian
\begin{equation}
\L^\mathrm{std}(\vU,p)=\int_{\Omega}\left(\W(\tF(\vU))+p\,C(J(\tF(\vU))) - \BG\cdot\vU\right)\,dV - \int_{\Gamma_N}\TG\cdot \vU\,dA.
\label{eq:standard-method-incompr}
\end{equation}

\subsubsection{Compatible-Strain mixed Finite Element Method (CSMFEM)}
\label{subsubsec:CSMFEM}
The \emph{CSMFEM} for 2D and 3D incompressible materials introduced in  \cite{Shojaei2018} and \cite{Shojaei2019}, respectively, is used to compare the \emph{NDTNS} method with a mixed formulation based on scalar and vector-valued finite elements for nonlinear elasticity. Therein, the displacement $\vU$, displacement gradient $\tK$, 1\textsuperscript{st}~Piola--Kirchhoff stress tensor $\tP$, and the pressure $p$ are taken as independent fields, which live in the function spaces $\VHone$, $\vH(\curl;\Omega, \R^d)^d$, $\vH(\div;\Omega, \R^d)^d$, and $\Ltwo$, respectively. Note, that $\vH(\curl;\Omega, \R^d)^d$ and $\vH(\div;\Omega, \R^d)^d$ are matrix valued where each row is in $\vH(\curl;\Omega, \R^d)$ and $\vH(\div;\Omega, \R^d)$, respectively. The Lagrangian in the incompressible case reads
\begin{align}
\L^{\mathrm{CSMFEM}}_{2D}(\vU,\tK,\tP,p)&=\int_{\Omega}\left(\W(\tK+\tI) + (\tK-\Grad\vU):\tP - p\, C(J(\tK))- \BG\cdot\vU\right)\,dV - \int_{\Gamma_N}\TG\cdot\vU\,dA,\\
\L^{\mathrm{CSMFEM}}_{3D}(\vU,\tK,\tP,p)&=\L^{\mathrm{CSMFEM}}_{2D}(\vU,\tK,\tP,p) + \int_{\Omega}\dfrac{\alpha}{2}(\tK-\Grad\vU):(\tK-\Grad\vU)\,dV,
\end{align}
where $\alpha \geq 0$ denotes a stability parameter supporting the equality $\tK=\Grad\vU$.  As suggested in \cite{Shojaei2019}, we set $\alpha =1 \times 10^6$ for all numerical 3D examples. 

In the two-dimensional case \cite{Shojaei2018} the authors propose the methods ``H1c1dn1L0'' for linear and ``H2c2dn2L1'' for quadratic displacement fields. More precisely, $(\vU,\tK,\tP,p)\in [U^k_h]^2\times [\mathcal{N}_{II}^k]^2\times [RT^{k-1}]^2\times Q_h^{k-1}$, $k=1,2$. This leads to 25 (55) total DoFs and 24 (42) coupling DoFs per triangle for the linear (quadratic) method. Further, using \eqref{eq:asymp_relation}, the asymptotic number of DoFs with respect to the number of vertices is 22\#V (64\#V) total and 20\#V (38\#V) coupling DoFs for linear (quadratic) elements. Compared with Table~\ref{tab:DoFs_combinations}, the \emph{CSMFEM} method has less total but more coupling DoFs than \emph{NDTNS}.

For the finite element discretization of $\vU$, $\tK$, $\tP$, and $p$ the following spaces are used for \emph{CSMFEM} in three dimensions: $\vU\in [U^2_h]^3$, $\tK\in [\mathcal{N}_h]^3$, where $\mathcal{N}_h=\mathcal{N}_{II}^1\oplus \Pi_{T\in\mathcal{T}}\mathcal{N}_{I}^{2,\mathrm{bubble}}(T)$, $\tP\in [RT^0]^3$, and $p\in Q_h^0$, leading to $88$ total and $78$ coupling DoFs per tetrahedron. $RT^k$ and $\mathcal{N}_{II}^k$ 
denote the degree-$k$ Raviart--Thomas \cite{RT77} and N\'ed\'elec \cite{Nedelec1986} finite element space, respectively. For $\mathcal{N}_h$ the N\'ed\'elec space of second kind gets enriched with the three cubic tangential bubble functions $\mathcal{N}_{I}^{2,\mathrm{bubble}}(T)$ appearing for the N\'ed\'elec space of first kind, which lies in between the full quadratic and cubic polynomial space, \cite{Zaglmayr06}. The number of asymptotic DoFs is 162\#V and 102\#V total and coupling, respectively. The 3D \emph{CSMFEM} method has about half of the coupling DoFs compared to the \emph{NDTNS} method, cf. Table~\ref{tab:DoFs_combinations}. Nevertheless, as we will show in Section~\ref{sec:Numerical results}, our method leads to improved convergence rates and more stable results.

\subsubsection{Simo--Taylor--Pister (STP) like method}
In \cite{SimoTaylorPister85} a three-field formulation for nearly incompressible materials has been proposed. A new unknown $\Theta$ is introduced as the dilatation field, which is coupled to the determinant of the deformation gradient by a Lagrange multiplier $\xi$. Then, the adapted isochoric deformation gradient according to the Flory split \cite{Flory61}
\begin{align*}
  \mathring{\tF}= \left(\frac{\Theta}{\det \tF}\right)^{1/d}\tF
\end{align*}
is used for the material law $\W(\cdot)$. We enforce the incompressibility constraint on $\Theta$ (instead of $\det \tF$) by the pressure $p$. Thus, the Lagrangian reads
\begin{align*}
  \L^{\mathrm{STP}}(\vU,\Theta,\xi,p)=\int_{\Omega}\left(\W(\mathring{\tF}) + \xi\left(\Theta-\det \tF\right) - p\,C(\Theta) - \BG\cdot\vU\right)\,dV - \int_{\Gamma_N}\TG\cdot\vU\,dA.
\end{align*}
For the finite element discretization, we follow \cite{schonherrRobustHybridMixed2022} and 
consider 
quadratic Lagrange finite elements for the displacement $\vU$ and element-wise constant discontinuous elements 
or for $\Theta$, $\xi$, and $p$. This method is denoted as \emph{STP}. 
As only the displacement field has coupling DoFs, the number is
12/30 per triangle/tetrahedron (and 15/33 total DoFs).

\subsection{Examples}
\label{subsec:Examples}
In the following, we provide a detailed numerical comparison of the proposed \emph{NDTNS} method with the above-mentioned methods. 
We focus on the case with quadratic displacement approximations throughout ($k=2$). All the numerical examples were performed in the open source finite element library NETGEN/NGSolve\footnote{\href{www.ngsolve.org}{www.ngsolve.org}} \cite{Sch97,Sch14}. For reproducibility, the code for the benchmarks and computational results are publicly available on \cite{FNSZ2025}.

We set $\beta=1$, i.e., no damping is required for Newton-Raphson. The pressure stabilization in \eqref{eq:pressure_reg} is set to $\varepsilon_p=\mu\times 10^{-7}$. We also use this stabilization for \emph{std} for the pressure, for \emph{STP} for the pressure, dilatation, and Lagrange multiplier, and for \emph{CSMFEM} for the pressure and 1\textsuperscript{st} Piola--Kirchhoff stress tensor. We stop Newton-Raphson if a residuum of $10^{-5}$ is reached. Further, we set $\Delta \xi_{\mathrm{init}}=0.1$, $\mathrm{tol}_{\mathrm{inc}}=10^{-5}$, and the maximal iterations to $n_{\max}=40$.

\subsubsection{Inflation of a cylindrical shell (2D)  and a hollow spherical ball
(3D)}
\label{subsubsec:Inflation shell}
We consider the inflation (expansion) of a cylindrical shell and a hollow spherical ball in two and three dimensions, respectively; see Figure~\ref{fig:Inflation_ball_geo}. The inner radius, outer radius, and Lam\'e parameter are in both cases given by \cite{Shojaei2018,Shojaei2019}
\begin{align*}
  R_{\mathrm{in}}=0.5\mathrm{mm},\qquad R_{\mathrm{out}}=1\mathrm{mm},\qquad \mu= 1 \mathrm{N/mm}^2.
\end{align*}
\begin{figure}[ht!]
	\centering
  \begin{tikzpicture}[scale=0.8]
    \draw[fill=lightgray] (0,0) circle (2);
    \draw[fill=white] (0,0) circle (1);
    \draw[->] (0,0) -- (1/1.41,1/1.41) node[midway, below right]{$R_{\mathrm{in}}$};
    \draw[->] (0,0) -- (1.41,-1.41) node[ left]{$R_{\text{out}}$};

    \foreach \x in {0, 30,60,90,120,150,180,210,240,270,300,330} {
        \draw[red, thick, ->] ({2*cos(\x)}, {2*sin(\x)}) -- ({2.4*cos(\x)}, {2.4*sin(\x)});
    }

    \draw[dashed] (-2,0)--(2,0);
    \draw[dashed] (0,-2)--(0,2);

    \node at (2., 2.) {$\vU_{\mathrm{out}}$};
  \end{tikzpicture}\hspace*{3cm}
	\includegraphics[width=0.3\linewidth]{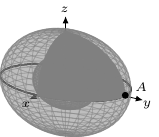}
    \caption{2D and 3D geometries for cylindrical shell and hollow spherical ball example, respectively.}
	\label{fig:Inflation_ball_geo}
\end{figure}
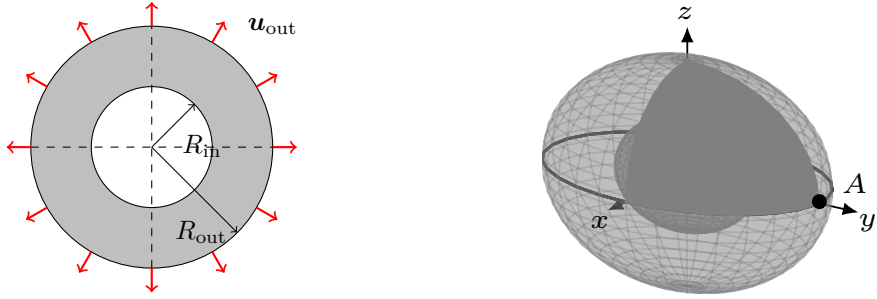
The inner boundary is traction-free and at the outer boundary the
displacement $\UG_{\mathrm{out}}=(\gamma-1)\vX$ is prescribed with $\vX$ the position vector and $\gamma > 1$.
Due to symmetry and incompressibility, the problems benefit from having exact
analytic solutions.
The exact displacement is
\begin{align*}
    \vU_{\mathrm{ex}}(\vX) = \left(\frac{r(R)}{R}-1\right)\vX,
\end{align*}
    where $ R=\|\vX\|$, $
    r(R) = \left(R^d+(\gamma^d-1)R^d_{\text{out}}\right)^{\frac{1}{d}}$,
    and $d\in \{2,3\}$ is the spatial dimension.
The exact pressure is 
\begin{align*}
    p_{\mathrm{ex}}(\vX)=\left\{
        \begin{tabular}{ll}
    $\mu\frac{R^2}{r(R)^2}-\frac{\mu(\gamma^2-1)R^2_{\mathrm{out}}}{2}\left(\frac{1}{r(R_{\mathrm{in}})^2}
    -\frac{1}{r(R)^2}\right)+\mu\ln\left(\frac{r(R_{\mathrm{in}})R}{R_{\mathrm{in}}r(R)}\right),
            $ & if $d=2$,\\[.2ex]
    $\mu\frac{R^4_{\text{in}}}{r^4(R_{\text{in}})}-\frac{\mu}{2}\left(g(R)-g(R_{\text{in}})\right),$
           & if $d=3$,
        \end{tabular}
        \right.
\end{align*}
where $g(R)=R(3r^3(R)+(\gamma^3-1)R^3_{\text{out}})/r^4(R)$. 
The exact deformation gradient is then $\tF_{\mathrm{ex}}=\Grad \vU_{\mathrm{ex}}+\tI$ and the 1\textsuperscript{st}~Piola--Kirchhoff stress tensor is computed by $\tP_{\mathrm{ex}}=\tilde{\tP}_{\mathrm{ex}}+p_{\mathrm{ex}}\,Q(\tF_{\mathrm{ex}})$, where $Q(\tF)= C^\prime(\tF)\tF^{-T}$. Due to symmetry, only one quarter, respectively, one eight of the domain is considered in two and three dimensions and symmetry boundary conditions are prescribed on the arising boundaries.

We conduct a convergence study of the methods presented in Subsection~ \ref{subsec:Methods} under successive mesh refinements, where the mesh elements are quadratically curved at the boundaries. 

For \emph{NDTNS}, we locally compute a postprocessed displacement \cite{FJQ2019} $\vU_h^*\in [Q_h^{k+1}]^d$ as follows: 
\begin{alignat*}{2}
    \int_T\Grad \vU_h^*:\Grad \mathbf{v}_h dV = &\;
    \int_T(\tF_h-\tI):\Grad \mathbf{v}_h dV, \quad  \forall \mathbf{v}_h\in
    [\Pol^{k+1}(T)]^d,&&\quad \forall T\in\TT,\\
    \int_T\vU_h^*dV = &\;
    \int_T \vU_h dV&&\quad \forall T\in\TT.
\end{alignat*}
The $L^2$-norms of the errors in 
$\vU, p,  
\vek{F},$ and $\vek{P}$ are reported in Table~\ref{tab:1} for the two-dimensional case and in Table~\ref{tab:2} for the three-dimensional case.
We take $\gamma =2$ in 2D and $\gamma = 1.5$ in 3D. 

In the 2D results (Table~\ref{tab:1}), we observe optimal third-order convergence for all variables $\vU_h, p_h, \vek{F}_h$, and $\vek{P}_h$, and 
a superconvergent fourth-order of convergence for the postprocessed displacement $\vU_h^*$ using \emph{NDTNS} without a stabilization parameter, $\tau=0$. Setting $\tau=100$ the convergence is at the beginning slightly reduced for $p_h$, $\tF_h$, $\tP_h$, and thus for the postprocessed displacement $\vU^\ast_h$. With finer meshes the super-convergence is regained. Using a strong and mesh-size dependent stabilization $\tau=100/h$ the super-convergence for $p_h$, $\tF_h$, $\tP_h$, and $\vU^\ast_h$ is lost. However, this strong type of stabilization is only required in regions where the deformation is massive. In moderate deformations, setting $\tau$ constant or even zero yields stable and accurate solutions. The investigation of optimal choices for the stabilization term and parameter will be the topic of future research.

For \emph{std}, we achieve third-order convergence in displacement and second-order convergence in the remaining variables.
The \emph{CSMFEM} method yields similar convergence rates to \emph{std}, with slightly larger error magnitudes.
The \emph{STP} method, on the other hand, produces second-order convergence in displacement and first-order convergence in the other variables. The reduced convergence is expected due to the piece-wise constant discretization of the pressure (and dilatation).

In the 3D results (Table~\ref{tab:2}), convergence rates for \emph{NDTNS} are similar to 2D, showing (nearly) optimal third-order of convergence for all variables $\vU_h, p_h, \vek{F}_h$, and $\vek{P}_h$, with 
(nearly)  fourth-order superconvergence for the postprocessed displacement $\vU_h^*$. The choice of the stabilization parameter $\tau$ affects the solution as in the 2D case.

The \emph{std} method achieves third-order convergence in displacement and (nearly) second-order convergence in the other variables. Both \emph{CSMFEM} and \emph{STP} methods display comparable accuracy, achieving second-order convergence in displacement and close-to-first-order convergence in the other variables.
The reduced accuracy for \emph{CSMFEM} in 3D relative to 2D is attributed to the use of lower-order finite element spaces for $p$ and $\tP$.
From both tables, we conclude that \emph{NDTNS} achieves the best accuracy among all methods in both 2D and 3D.

\begin{table}[tb]
\centering
    \resizebox{0.95\textwidth}{!}{
\begin{tabular}{ll ll ll ll ll ll}
\toprule
    method&   $h$   & $\|\vU-\vU_h\|$  & e.o.c
&  $\|p-p_h\|$& e.o.c    &  $\|\vek{F}-\vek{F}_h\|$& e.o.c   
    & $\|\vP-\vP_{h}\|$  & e.o.c  
    & $\|\vU-\vU_h^*\|$  & e.o.c 
    \\ \midrule
 & 0.25 &  7.43e-04 & - & 1.74e-03 & - & 5.46e-03 & - & 7.99e-03 & - & 1.11e-04 & - \\
\emph{NDTNS} $\tau=0$ & 0.25/2  & 1.08e-04 & 2.78 & 2.26e-04 & 2.95 & 8.19e-04 & 2.74 & 1.19e-03 & 2.75 & 7.65e-06 & 3.85 \\
 & 0.25/4  & 1.45e-05 & 2.9 & 3.17e-05 & 2.83 & 1.12e-04 & 2.87 & 1.66e-04 & 2.84 & 5.07e-07 & 3.92 \\
 & 0.25/8  & 1.88e-06 & 2.95 & 4.08e-06 & 2.96 & 1.47e-05 & 2.93 & 2.18e-05 & 2.93 & 3.28e-08 & 3.95 \\
 \midrule
 & 0.25 & 9.11e-04 & - & 3.57e-02 & - & 2.68e-02 & - & 9.44e-02 & - & 7.26e-04 & - \\
\emph{NDTNS} $\tau=100$ & 0.25/2 & 1.21e-04 & 2.91 & 5.47e-03 & 2.71 & 5.84e-03 & 2.2 & 1.56e-02 & 2.59 & 8.12e-05 & 3.16 \\
 & 0.25/4 & 1.50e-05 & 3.02 & 7.72e-04 & 2.83 & 1.05e-03 & 2.48 & 2.38e-03 & 2.72 & 7.26e-06 & 3.48 \\
 & 0.25/8 & 1.87e-06 & 3.0 & 1.10e-04 & 2.81 & 1.64e-04 & 2.68 & 3.53e-04 & 2.75 & 5.65e-07 & 3.68 \\
\midrule
 & 0.25 & 1.34e-03 & - & 1.33e-01 & - & 3.95e-02 & - & 3.74e-01 & - & 1.28e-03 & - \\
\emph{NDTNS} $\tau=100/h$ & 0.25/2 & 1.88e-04 & 2.84 & 3.42e-02 & 1.96 & 1.13e-02 & 1.81 & 9.69e-02 & 1.95 & 1.79e-04 & 2.84 \\
 & 0.25/4 & 2.30e-05 & 3.04 & 9.04e-03 & 1.92 & 2.93e-03 & 1.94 & 2.57e-02 & 1.91 & 2.15e-05 & 3.05 \\
 & 0.25/8 & 2.85e-06 & 3.01 & 2.40e-03 & 1.91 & 7.47e-04 & 1.97 & 6.92e-03 & 1.89 & 2.66e-06 & 3.02 \\
 \midrule 
 & 0.25 & 1.49e-03 & - & 3.06e-02 & - & 5.46e-02 & - & 1.15e-01 & - &   &   \\
\emph{CSMFEM} & 0.25/2  & 2.14e-04 & 2.8 & 8.82e-03 & 1.79 & 1.53e-02 & 1.83 & 3.31e-02 & 1.79 &   &   \\
 & 0.25/4  & 3.06e-05 & 2.8 & 2.66e-03 & 1.73 & 4.10e-03 & 1.9 & 1.00e-02 & 1.72 &   &   \\
 & 0.25/8  & 5.05e-06 & 2.6 & 6.81e-04 & 1.96 & 1.05e-03 & 1.97 & 2.87e-03 & 1.8 &   &   \\
 \midrule
 & 0.25 & 1.33e-03 & - & 8.05e-03 & - & 4.54e-02 & - & 5.06e-02 & - &   &   \\
\emph{std} & 0.25/2  & 1.82e-04 & 2.87 & 1.64e-03 & 2.29 & 1.23e-02 & 1.89 & 1.33e-02 & 1.92 &   &   \\
 & 0.25/4  & 2.37e-05 & 2.94 & 3.39e-04 & 2.28 & 3.19e-03 & 1.94 & 3.44e-03 & 1.95 &   &   \\
 & 0.25/8  & 3.02e-06 & 2.97 & 8.08e-05 & 2.07 & 8.16e-04 & 1.97 & 8.77e-04 & 1.97 &   &   \\
 \midrule
 & 0.25 & 6.44e-03 & - & 9.97e-02 & - & 1.11e-01 & - & 3.42e-01 & - &   &   \\
\emph{STP} & 0.25/2  & 1.35e-03 & 2.26 & 2.94e-02 & 1.76 & 4.86e-02 & 1.2 & 1.06e-01 & 1.69 &   &   \\
 & 0.25/4  & 3.21e-04 & 2.07 & 1.19e-02 & 1.3 & 2.43e-02 & 1.0 & 4.67e-02 & 1.18 &   &   \\
 & 0.25/8  & 7.96e-05 & 2.01 & 5.54e-03 & 1.11 & 1.23e-02 & 0.99 & 2.24e-02 & 1.06 &   &   \\
    \bottomrule
\end{tabular}
    }
\vspace{.3ex}
\caption{History of convergence for different methods with quadratic displacement approximations ($k=2$)
    for inflation of a cylindrical shell (2D), $\gamma=2$, with estimated order of convergence (e.o.c.).}
\label{tab:1}
\end{table}

\begin{table}[tb]
\centering
    \resizebox{0.95\textwidth}{!}{
\begin{tabular}{ll ll ll ll ll ll}
\toprule
    method&   $h$   & $\|\vU-\vU_h\|$  & e.o.c
&  $\|p-p_h\|$& e.o.c    &  $\|\vek{F}-\vek{F}_h\|$& e.o.c   
    & $\|\vP-\vP_{h}\|$  & e.o.c  
    & $\|\vU-\vU_h^*\|$  & e.o.c 
    \\ \midrule
 & 0.25 & 1.11e-03 & - & 2.75e-03 & - & 8.50e-03 & - & 1.53e-02 & - & 2.62e-04 & - \\
\emph{NDTNS} $\tau=0$ & 0.25/2  & 1.88e-04 & 2.56 & 4.68e-04 & 2.55 & 1.50e-03 & 2.5 & 2.87e-03 & 2.42 & 2.34e-05 & 3.49 \\
 & 0.25/4 &  2.52e-05 & 2.9 & 6.23e-05 & 2.91 & 2.12e-04 & 2.83 & 4.17e-04 & 2.78 & 1.62e-06 & 3.85 \\
 \midrule
 & 0.25 & 2.10e-03 & - & 6.90e-02 & - & 4.92e-02 & - & 2.47e-01 & - & 2.04e-03 & - \\
\emph{NDTNS} $\tau=100$ & 0.25/2 & 2.69e-04 & 2.97 & 1.23e-02 & 2.49 & 1.20e-02 & 2.04 & 4.82e-02 & 2.36 & 2.47e-04 & 3.04 \\
 & 0.25/4 & 2.68e-05 & 3.32 & 1.98e-03 & 2.64 & 2.42e-03 & 2.31 & 8.52e-03 & 2.5 & 2.16e-05 & 3.52 \\
 \midrule
 & 0.25 & 3.47e-03 & - & 2.17e-01 & - & 7.41e-02 & - & 8.06e-01 & - & 3.46e-03 & - \\
\emph{NDTNS} $\tau=100/h$ & 0.25/2 & 5.85e-04 & 2.57 & 6.51e-02 & 1.74 & 2.23e-02 & 1.73 & 2.53e-01 & 1.67 & 5.82e-04 & 2.57 \\
 & 0.25/4 & 6.91e-05 & 3.08 & 1.83e-02 & 1.83 & 5.73e-03 & 1.96 & 7.53e-02 & 1.75 & 6.85e-05 & 3.09 \\
 \midrule
 & 0.25  & 5.45e-03 & - & 1.69e-01 & - & 1.15e-01 & - & 4.53e-01 & - &   &   \\
\emph{CSMFEM} & 0.25/2  & 1.21e-03 & 2.17 & 8.55e-02 & 0.99 & 5.12e-02 & 1.17 & 2.30e-01 & 0.98 &   &   \\
 & 0.25/4  & 2.98e-04 & 2.03 & 4.75e-02 & 0.85 & 2.41e-02 & 1.09 & 1.37e-01 & 0.75 &   &   \\
 \midrule
 & 0.25  & 2.45e-03 & - & 1.26e-02 & - & 6.91e-02 & - & 8.23e-02 & - &   &   \\
\emph{std} & 0.25/2  & 4.11e-04 & 2.58 & 3.19e-03 & 1.98 & 2.18e-02 & 1.66 & 2.50e-02 & 1.72 &   &   \\
 & 0.25/4  & 5.46e-05 & 2.91 & 6.11e-04 & 2.39 & 6.09e-03 & 1.84 & 6.69e-03 & 1.9 &   &   \\
 \midrule
  & 0.25 &  5.11e-03 & - & 2.71e-01 & - & 1.09e-01 & - & 9.07e-01 & - &   &   \\
\emph{STP} & 0.25/2  & 1.35e-03 & 1.92 & 1.23e-01 & 1.14 & 4.92e-02 & 1.15 & 4.76e-01 & 0.93 &   &   \\
 & 0.25/4 & 3.16e-04 & 2.09 & 7.84e-02 & 0.65 & 2.25e-02 & 1.13 & 3.40e-01 & 0.48 &   &   \\
\bottomrule
\end{tabular}
    }
\vspace{.3ex}
\caption{History of convergence for different methods with quadratic displacement approximations ($k=2$)
    for inflation of a hollow sphere (3D), $\gamma=1.5$, with estimated order of convergence (e.o.c.).}
\label{tab:2}
\end{table}

\subsubsection{Cook's membrane}
\label{subsubsec:Cooks membrane}

\begin{figure}[ht!]
	\centering
	\begin{tikzpicture}[scale=0.7]
    \draw[thick, fill=lightgray] (0,0) -- (4.8,4.4) -- (4.8,6.0) -- (0,4.4) -- cycle;
    \draw[fill=darkgray] (0,-0.2) -- (0,4.6) -- (-0.4,4.6) -- (-0.4,-0.2) -- cycle;
    \draw[dashed] (-0.6,4.4) -- (4.8,4.4);
    \draw[dashed] (-0.6,6) -- (4.8,6);
    \draw[dashed] (4.8,-0.2) -- (4.8,4.4);

    \draw[<->] (-0.6,0) -- (-0.6,4.4) node[midway, left] {$44$mm};
    \draw[<->] (-0.6,4.4) -- (-0.6,6.0) node[midway, left] {$16$mm};
    \draw[<->] (0,-0.2) -- (4.8,-0.2) node[midway, below] {$48$mm};
    \draw[->] (5.2,4.8) -- (5.2,5.6) node[midway, right] {$\TG$};
    \draw[fill] (4.8,6.0) circle [radius=0.1] node[above] {$A$};
  \end{tikzpicture}\hspace*{2cm}
  \begin{tikzpicture}[scale=0.7,line join = round]
    \draw[dashed] (0,0,0) -- (0,4.4,0) ;
    \draw[thick,fill=lightgray, opacity=0.5] (0,0,1) -- (4.8,4.4,1) -- (4.8,6.0,1) -- (0,4.4,1) -- cycle;

    \draw[thick,fill=lightgray, opacity=0.5] (4.8,4.4,1) -- (4.8,4.4,0) -- (4.8,6.0,0) -- (4.8,6.0,1) -- cycle;
    \draw[thick,fill=lightgray, opacity=0.5] (4.8,6.0,1) -- (4.8,6.0,0) -- (0,4.4,0)  -- (0,4.4,1) -- cycle;

    \draw[fill=darkgray] (0,-0.2,1) -- (0,4.6,1) -- (-0.4,4.6,1) -- (-0.4,-0.2,1) -- cycle;
    \draw[fill=darkgray] (0,4.6,1) --(0,4.6,0) -- (-0.4,4.6,0) -- (-0.4,4.6,1) -- cycle;
    \draw[fill=darkgray] (0,4.6,1) --(0,4.6,0) -- (0,4.4,0) -- (-0,4.4,1) -- cycle;
    \draw[fill=darkgray] (0,-0.2,1) --(0,-0.2,0) -- (0,0,0) -- (0,-0,1) -- cycle;

    \draw[fill] (4.8,6.0,1) circle [radius=0.1] node[below left] {$A$};
    \draw[<->] (-0.7,4.4,1) -- (-0.7,4.4,0) node[midway, above left] {$10$mm};
  \end{tikzpicture}
	
	\caption{Geometry of 2D (left) and extended 3D (right) Cook's membrane examples.}
	\label{fig:cooks_membrane_geo}
\end{figure}
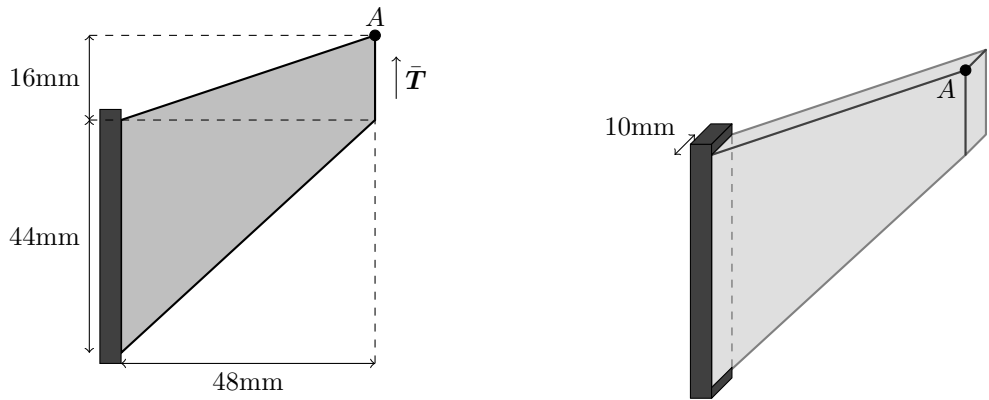
Next, we consider the Cook's membrane problem proposed in \cite{Shojaei2018,Shojaei2019} in two and three dimensions. The membrane is clamped on the left and a shear force is applied at the right boundary. We consider $\mu=1$ N/mm$^2$. The force acts vertically, $\TG=0.5 \vek{e}_y$, in both examples. 
The deflection of point $A$ is the quantity of interest. The geometry can be found in Figure~\ref{fig:cooks_membrane_geo} (left for 2D, right for 3D). We use structured meshes, where $n\times n$ quadrilaterals are divided into triangles in 2D with $n=4$, $8$, $16$, and $32$. In 3D, we exploit the symmetry in the $z$-direction and use one layer. The singularity at the top left vertex (edge) in 2D (3D), appearing due to the change of boundary conditions, is known to prevent satisfying results for the standard method on coarse meshes. In both configurations, we choose the stabilization $\tau = 100\mu/h$ for \emph{NDTNS} in the region $x^2+(y-0.44)^2<0.1^2$ around the top left singularity and $\tau=100\mu$ elsewhere.

In Figure~\ref{fig:cook2d_PK} and Figure~\ref{fig:cook3d_PK} the norm of the 1\textsuperscript{st}~Piola--Kirchhoff stress tensor on the deformed configuration is presented for the four methods for the $n=16$ grid. One can clearly observe the singularity at the top left corner. Around this corner, the elements get heavily compressed and rotated. The \emph{std} and \emph{CSMFEM} methods produce a distorted mesh around the singularity and both methods fail to converge on the finest $n=32$ mesh. The \emph{STP} method's solution always converges. However, it is less smooth in 2D. In 3D the \emph{std} and \emph{CSMFEM} solutions have distorted meshes and again fail to converge on the finest level. The \emph{NDTNS} method produces appealing deformations in both settings.

In Table~\ref{tab:cook} the deflection at point $A$ is displayed for different meshes. The values of all methods are in agreement when the methods converge.

\begin{figure}[ht!]
    \centering
    \includegraphics[width=0.9\linewidth]{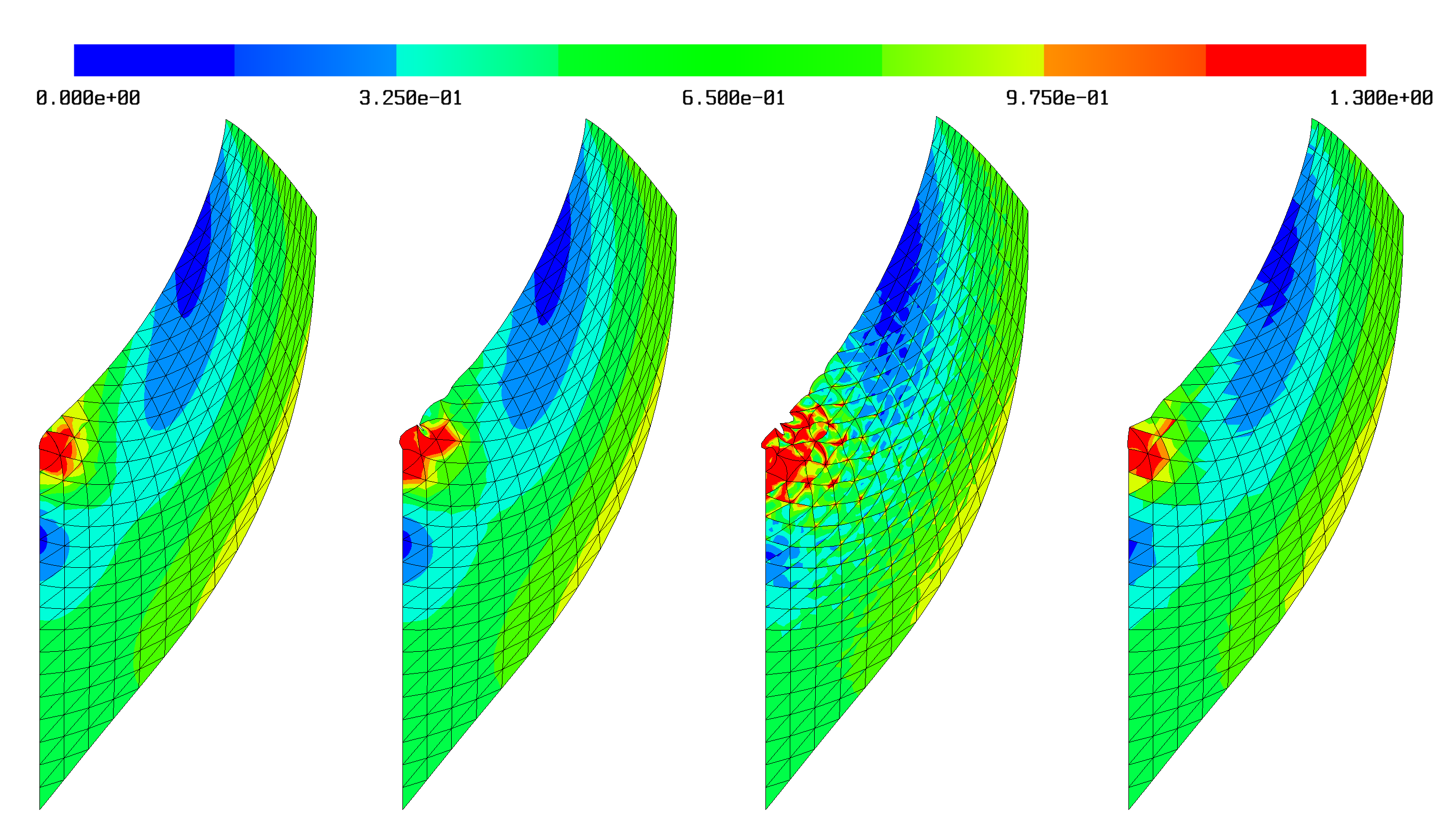}
  \caption{Norm of 1\textsuperscript{st}~Piola--Kirchhoff stress tensor of 2D Cook's membrane on final configuration. $n=16$.  From left to right: \emph{NDTNS}, \emph{std}, \emph{CSMFEM}, \emph{STP}.}
    \label{fig:cook2d_PK}
\end{figure}

\begin{figure}[ht!]
    \centering
    \includegraphics[width=0.9\linewidth]{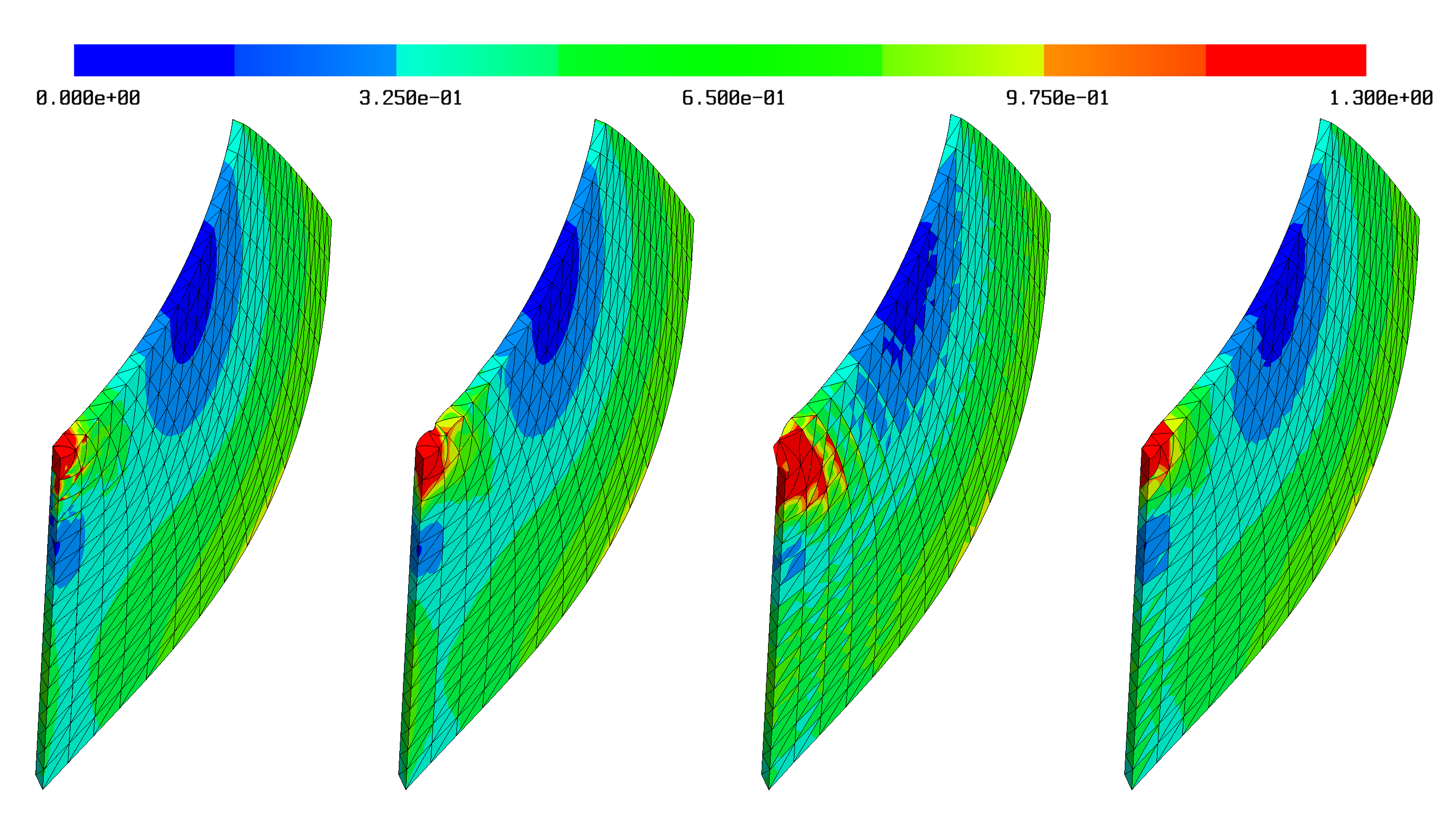}
    
    \caption{Norm of 1\textsuperscript{st}~Piola--Kirchhoff stress tensor of 3D Cook's membrane on final configuration. $n=16$. From left to right: \emph{NDTNS}, \emph{std}, \emph{CSMFEM}, \emph{STP}.}
    \label{fig:cook3d_PK}
\end{figure}

\begin{table}[tb]
\centering
\begin{tabular}{ll ll}
\toprule
    method&   $n$   & $u_x(A)$  & $u_y(A)$
    \\ \midrule
 & 4 & -0.24939 & 0.24071 \\
\emph{NDTNS} & 8 & -0.25142 & 0.24168 \\
 & 16 & -0.25243 & 0.24228 \\
 & 32 & -0.25316 & 0.24276 \\
 \midrule
 & 4 & -0.25227 & 0.24263 \\
\emph{CSMFEM} & 8 & -0.25387 & 0.24305 \\
 & 16 & -0.27156 & 0.24608 \\
 & 32 & F 0.73 &  \\
  \midrule
 & 4 & -0.25264 & 0.24172 \\
\emph{std} & 8 & -0.25438 & 0.24273 \\
 & 16 & -0.25623 & 0.24325 \\
 & 32 & F 0.95 &  \\
  \midrule
 & 4 & -0.26769 & 0.24678 \\
\emph{STP} & 8 & -0.25977 & 0.24458 \\
 & 16 & -0.25635 & 0.24378 \\
 & 32 & -0.25509 & 0.24358 \\
 \bottomrule
\end{tabular}\hspace{2cm}
\begin{tabular}{ll ll}
\toprule
    method&   $n$   & $u_x(A)$  & $u_y(A)$
    \\ \midrule
 & 4 & -0.26665 & 0.25303 \\
\emph{NDTNS} & 8 & -0.27389 & 0.25699 \\
 & 16 & -0.27639 & 0.25884 \\
 & 32 & -0.27734 & 0.25970 \\
  \midrule
 & 4 & -0.28144 & 0.26222 \\
\emph{CSMFEM} & 8 & -0.28506 & 0.26370 \\
 & 16 & -0.28713 & 0.26492 \\
 & 32 & F 0.78 &  \\
  \midrule
 & 4 & -0.27051 & 0.25558 \\
\emph{std} & 8 & -0.27527 & 0.25816 \\
 & 16 & -0.27740 & 0.25937 \\
 & 32 & F 0.51 &  \\
  \midrule
 & 4 & -0.27381 & 0.25722 \\
\emph{STP} & 8 & -0.27751 & 0.25947 \\
 & 16 & -0.27819 & 0.26003 \\
 & 32 & -0.27843 & 0.26036 \\
\bottomrule
\end{tabular}
\vspace{.3ex}
\caption{2D (left) and 3D (right) Cook's membrane deflection of point $A$. F denotes that the method failed to converge at the given load step in $[0,1]$.}
\label{tab:cook}
\end{table}

\subsubsection{Compression of a 3D incompressible block}
\label{subsubsec:Cpmpression block}

We consider a block under inhomogeneous compression \cite{Shojaei2019}, adapted to the full incompressible case, depicted in Figure~\ref{fig:compr_block}. The material parameter is chosen to be $\mu=80.194$N/mm$^2$. At the top boundary the horizontal deflections are constrained to be zero, whereas at the bottom boundary the vertical displacement is set to zero. The other boundaries are traction-free. Due to symmetry, we consider a quarter of the block and prescribe symmetry boundary conditions at the arising boundaries. Many finite elements suffer from hourglass instabilities for this example \cite{Reese07}. 
Nominal vertical traction forces $\TG=(0,0, -f)$, with $f$ up to $1000$ N/mm$^2$, are considered such that about 90\% compression is achieved, which makes this example extremely challenging. Further, due to the inhomogeneous compression, singularities occur at the edges, where the forces start. The vertical deflection at point $A=(0,0,1)$ is the quantity of interest. 

Due to the strong compression, we choose $\tau=800\mu/h$ at the area of compression with a surrounding layer of $0.1$ for \emph{NDTNS}. On the remaining, less compressed part we set $\tau=800\mu$.
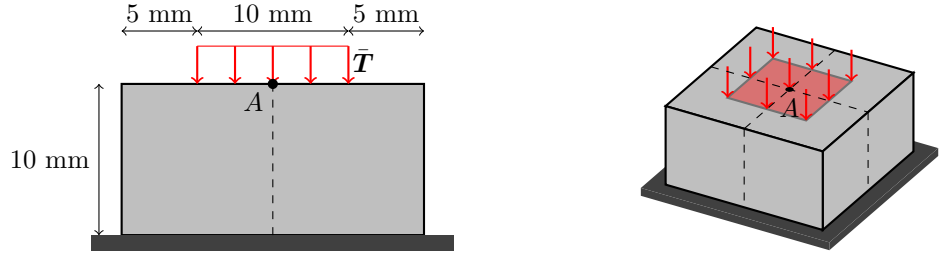
\begin{figure}[ht!]
	\centering

   \begin{tikzpicture}[scale=2]
     \fill[lightgray] (0, 0) rectangle (2, 1);
     \draw[thick] (0, 0) rectangle (2, 1);
    
     \fill[darkgray] (-0.2, -0.1) rectangle (2.2, 0);
    
     \draw[red,-] (0.5,1.25) -- (1.5,1.25);
     \foreach \x in {0.5, 0.75, 1.0, 1.25, 1.5} {
         \draw[red, thick, ->] (\x, 1.25) -- (\x, 1);
     }
    
     \node at (1, 1) [circle, fill=black, inner sep=1pt] {};
     \draw[fill] (1,1) circle [radius=0.03] node[below left] {$A$};
    
     \draw[<->] (-0.15, 0) -- (-0.15, 1) node[midway, left] {10 mm};
     \draw[<->] (0, 1.35) -- (0.5, 1.35) node[midway, above] {5 mm};
     \draw[<->] (0.5, 1.35) -- (1.5, 1.35) node[midway, above] {10 mm};
     \draw[<->] (1.5, 1.35) -- (2, 1.35) node[midway, above] {5 mm};

     \draw[dashed] (1,0) --(1,1);
    
     \node at (1.6, 1.15) {$\TG$};
 \end{tikzpicture}\hspace*{2cm}
\tdplotsetmaincoords{60}{120}
\begin{tikzpicture}[tdplot_main_coords,scale=1.2]
  \fill[darkgray] (-0.2,-0.2,0) -- (2.2,-0.2,0) -- (2.2,2.2,0) -- (-0.2,2.2,0) -- cycle;
  \fill[darkgray] (-0.2,2.2,-0.1) -- (2.2,2.2,-0.1) -- (2.2,2.2,0) -- (-0.2,2.2,0) -- cycle;
  \fill[darkgray] (2.2,2.2,-0.1) -- (2.2,-0.2,-0.1) -- (2.2,-0.2,0) -- (2.2,2.2,0) -- cycle;
  
  \draw[thick, fill=lightgray] (0,0,1) -- (2,0,1) -- (2,2,1) -- (0,2,1) -- cycle; 
  \draw[thick, fill=red, opacity=0.4] (0.5,0.5,1) -- (1.5,0.5,1) -- (1.5,1.5,1) -- (0.5,1.5,1) -- cycle;
  \draw[thick, fill=lightgray] (0,2,0) -- (0,2,1) -- (2,2,1) -- (2,2,0) -- cycle; 
  \draw[thick, fill=lightgray] (2,0,0) -- (2,2,0) -- (2,2,1) -- (2,0,1) -- cycle;

  \foreach \x in {0.5, 1.0, 1.5}
      \foreach \y in {0.5, 1.0, 1.5}
      {
          \draw[red, thick, ->] (\x,\y,1.4) -- (\x,\y,1);
      }
  
  \draw[dashed] (0,1,1) -- (2, 1, 1);
  \draw[dashed] (1,0,1) -- (1, 2, 1);
  \draw[dashed] (2,1,1) -- (2, 1, 0);
  \draw[dashed] (1,2,1) -- (1, 2, 0);

  \draw[fill] (1,1,1) circle [radius=0.04] node[below] {$A$};
\end{tikzpicture}
	\caption{Geometry of compression of block example.}
	\label{fig:compr_block}
\end{figure}

We consider three sets of unstructured meshes with 1590, 4189, and 13219 elements. Both \emph{NDTNS} and \emph{STP} methods produced stable results till the final load with $f=1000$. 
However, \emph{std} failed between  $f\in (300, 430)$, and \emph{CSMFEM} failed between $f\in (550, 600)$ depending on the mesh, see Table~\ref{tab:block3d}. 

\begin{table}[tb]
\centering
\begin{tabular}{lll}
\toprule
method  &  ne   &  $u_z(A)$  \\
\midrule
 & 1590 & -0.76756 \\
\emph{NDTNS} & 4189 & -0.76270 \\
 & 13219 & -0.76175 \\
 \midrule
 & 1590 & -0.77941 \\
\emph{CSMFEM} & 4189 & -0.76770 \\
 & 13219 & -0.76425 \\
 \midrule
 & 1590 & F, 0.941 \\
\emph{std} & 4189 & F, 0.754 \\
 & 13219 & -0.76007 \\
 \midrule
 & 1590 & -0.76428 \\
\emph{STP} & 4189 & -0.76145 \\
 & 13219 & -0.76125 \\
\bottomrule
\end{tabular}
\hspace{2cm}
\begin{tabular}{lll}
\toprule
method  &  ne   &  $u_z(A)$  \\
\midrule
 & 1590 & -0.90877 \\
\emph{NDTNS} & 4189 & -0.90144 \\
 & 13219 & -0.90085 \\
 \midrule
 & 1590 & F, 0.550 \\
\emph{CSMFEM} & 4189 & F, 0.554 \\
 & 13219 & F, 0.592 \\
 \midrule
 & 1590 & F, 0.376 \\
\emph{std} & 4189 & F, 0.302 \\
 & 13219 & F, 0.428 \\
 \midrule
 & 1590 & -0.89263 \\
\emph{STP} & 4189 & -0.90342 \\
 & 13219 & -0.90171 \\
\bottomrule
\end{tabular}
\vspace{.3ex}
\caption{3D incompressible block deflection at point $A$. Left: $\TG = (0,0,-400)$. Right: $\TG = (0,0,-1000)$. F denotes that the method failed to converge at the given load step in $[0,1]$. ne denotes the number of mesh elements.}
\label{tab:block3d}
\end{table}

The reason for failure is the appearance of a negative Jacobian determinant. We plot in Figure~\ref{fig:comp1} the Jacobian determinant on the deformed domain for load $f=400$ for the three stable methods on the mesh with 4189 elements. 
It is observed that all three methods produce similar deformations. \emph{NDTNS} produces the smallest incompressibility error, with the point-wise
Jacobian determinant $J_h=\det(\tF_h)\in[0.951,1.126]$, cf. Table~\ref{tab:block3d_det}. Note that the auxiliary deformation gradient field $\tF_h$ is considered, where the incompressibility condition is enforced by \eqref{eq:NDTNS-strong-formulation_c}. We also measure the $L^2$-projection onto the element-wise constants $J_{0,h}=\Pi^0_{L^2}(\det(\tF_h))\in[0.996,1.009]$, which smooths out possible outliers in small regions.
\emph{STP} method has $J_h=\det(\Grad \vU_h+\tI)\in [0.110,1.873]$ and $J_{0,h}=\Pi^0_{L^2}(\det(\Grad \vU_h+\tI))\in [0.9451,1.061]$. This clearly shows that the incompressibility constraint is not fulfilled point-wise but closely in a weak sense.  
\emph{CSMFEM} has strong point-wise deviations with $J_h=\det(\tK_h+\tI)\in [-2.721,8.887]$, however, the projected determinant $J_{0,h}=\Pi^0_{L^2}(\det(\tK_h+\tI))\in [0.823,1.192]$ is acceptable. For \emph{std}, both quantities reach zero and below so that the method fails.

Due to the large incompressibility errors, the \emph{CSMFEM} method does not produce stable results when the load increases to $f=1000$. Hence, we plot in Figure~\ref{fig:comp2} the Jacobian determinant on the deformed domain for the final load $f=1000$ for the methods \emph{NDTNS}
and \emph{STP} on the mesh with 4189 elements. It is observed that \emph{NDTNS} still has a better incompressibility error, with $J_h\in [0.503,1.406]$ and $J_{0,h}=[0.9733,1.025]$, 
than \emph{STP}, which produces $J_h\in [-1.189,4.140]$ and $J_{0,h}\in[0.7241,1.294]$.

It is also clear that larger incompressibility errors occur near the singular edges on the top boundary. In Figure~\ref{fig:comp3} and Figure~\ref{fig:comp4}, the norm of the 1\textsuperscript{st}~Piola--Kirchhoff stress tensor is displayed for final loads $f=400$ and $f=1000$, respectively.

\begin{table}[tb]
\centering

\resizebox{0.49\textwidth}{!}{
\begin{tabular}{llll}
\toprule
method  &  ne   &  $J(u/F)$  &  $J_0(u/F)$ \\
\midrule
 & 1590 &  (0.946, 1.063)  &  (0.9961, 1.004) \\
\emph{NDTNS} & 4189 &  (0.951, 1.126)  &  (0.996, 1.009) \\
 & 13219 &  (0.959, 1.042)  &  (0.9962, 1.005) \\
 \midrule
 & 1590 &  (-1.940, 7.068)  &  (0.7938, 1.117) \\
\emph{CSMFEM} & 4189 &  (-2.721, 8.887)  &  (0.823, 1.192) \\
 & 13219 &  (-2.192, 6.224)  &  (0.8845, 1.098) \\
 \midrule
 & 1590 &  (-1.357, 3.020)  &  (0.0, 1.511) \\
\emph{std} & 4189 &  (-1.089, 2.015)  &  (0.0, 1.305) \\
 & 13219 &  (-1.807, 2.937)  &  (0.0756, 1.597) \\
 \midrule
 & 1590 &  (-0.162, 2.162)  &  (0.9497, 1.071) \\
\emph{STP} & 4189 &  (0.110, 1.873)  &  (0.9451, 1.061) \\
 & 13219 &  (0.017, 1.834)  &  (0.9187, 1.061) \\
\bottomrule
\end{tabular}}
\resizebox{0.489\textwidth}{!}{
\begin{tabular}{llll}
\toprule
method  &  ne   &  $J(u/F)$  &  $J_0(u/F)$ \\
\midrule
 & 1590 &  (0.575, 1.459)  &  (0.9689, 1.043) \\
\emph{NDTNS} & 4189 &  (0.503, 1.406)  &  (0.9733, 1.025) \\
 & 13219 &  (0.411, 1.467)  &  (0.9671, 1.029) \\
 \midrule
 & 1590 &  (-17.68, 25.72)  &  (0.0001, 1.831) \\
\emph{CSMFEM} & 4189 &  (-17.38, 28.94)  &  (0.0005, 1.671) \\
 & 13219 &  (-12.43, 25.33)  &  (0.2101, 1.450) \\
 \midrule
 & 1590 &  (-1.357, 3.020)  &  (0.0, 1.511) \\
\emph{std} & 4189 &  (-1.089, 2.015)  &  (0.0, 1.305) \\
 & 13219 &  (-2.121, 3.131)  &  (0.0, 1.689) \\
 \midrule
 & 1590 &  (-1.981, 4.041)  &  (0.4498, 1.290) \\
\emph{STP} & 4189 &  (-1.189, 4.140)  &  (0.7241, 1.294) \\
 & 13219 &  (-0.845, 3.311)  &  (0.7028, 1.318) \\
\bottomrule
\end{tabular}}
%\vspace{.3ex}
\caption{3D incompressible block minimal and maximal values of determinant computed by the deformation gradient or displacement $J(u/F)$, and the $L^2$-projection onto element-wise constants $J_0(u/F)$. Left: $\TG = (0,0,-400)$. Right: $\TG = (0,0,-1000)$. For methods that fail to converge, the values at the last converged configuration are displayed. ne denotes the number of mesh elements.}
\label{tab:block3d_det}
\end{table}

\begin{figure}[ht!]
    \centering
    \includegraphics[width=1\linewidth]{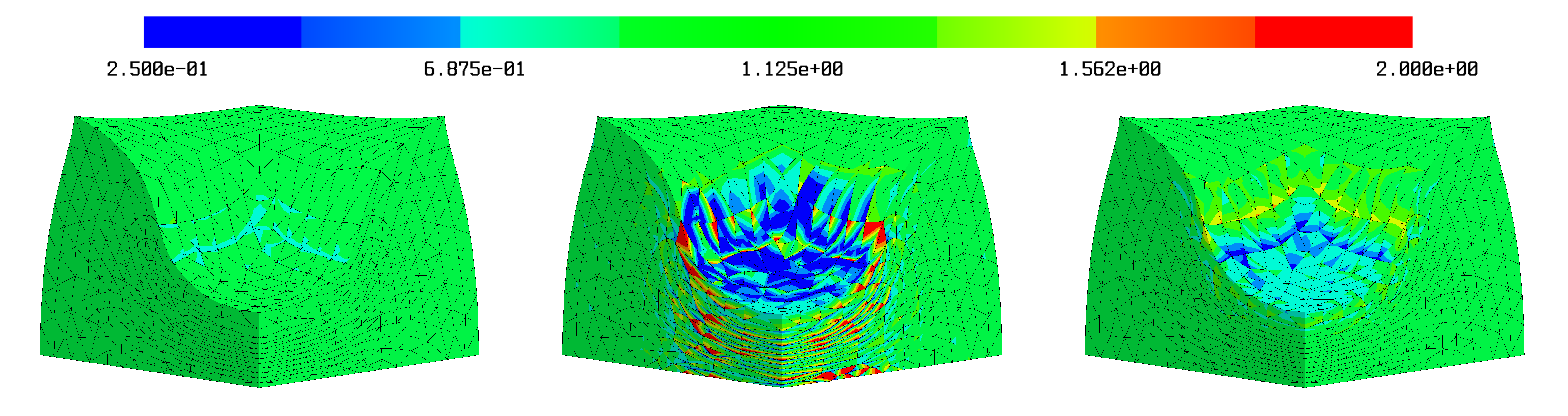}
    
    \caption{Jacobian determinant $J=\det(\tF)$ on deformed domain for block compressible problem with load $T=(0,0,-400)$. Left: \emph{NDTNS}. Middle: \emph{CSMFEM}. Right: \emph{STP}.}
    \label{fig:comp1}
\end{figure}

\begin{figure}[ht!]
    \centering
    \includegraphics[width=0.8\linewidth]{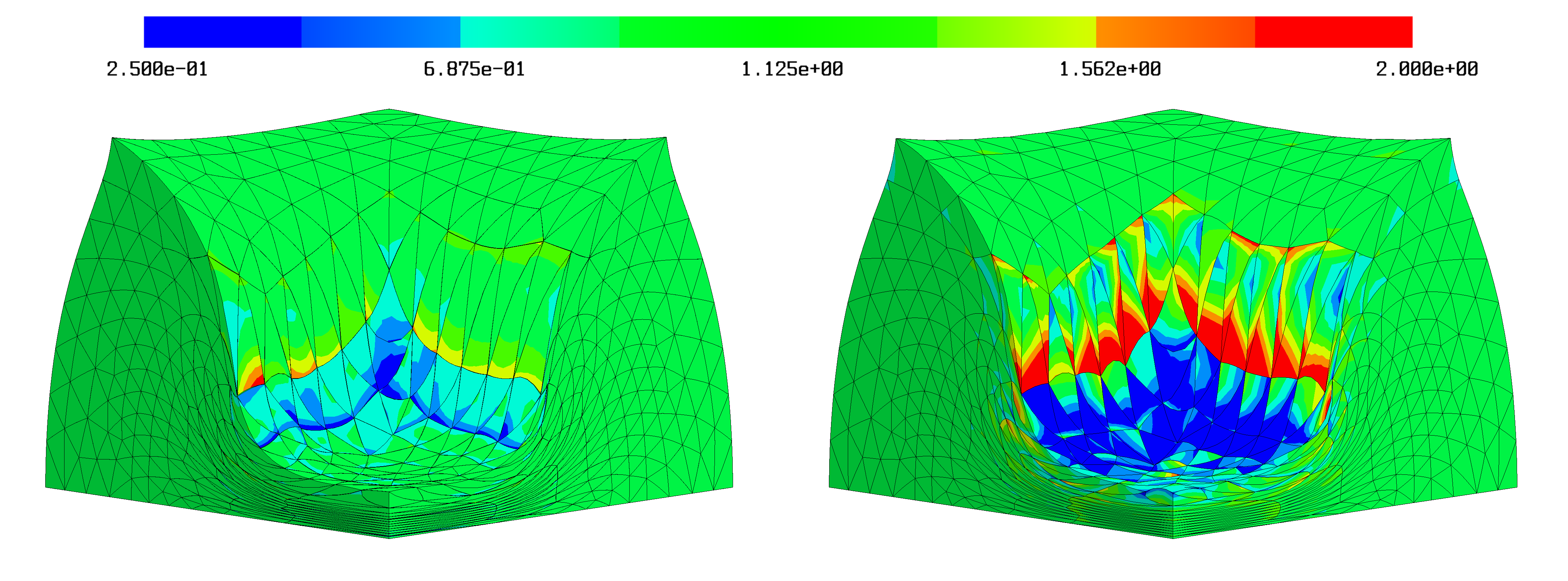}
   
    \caption{Jacobian determinant $J=\det(\tF)$ on deformed domain for block compressible problem
    with load $\vT=(0,0,-1000)$. Left: \emph{NDTNS}. Right: \emph{STP}.}
    \label{fig:comp2}
\end{figure}

\begin{figure}[ht!]
    \centering
    \includegraphics[width=1\linewidth]{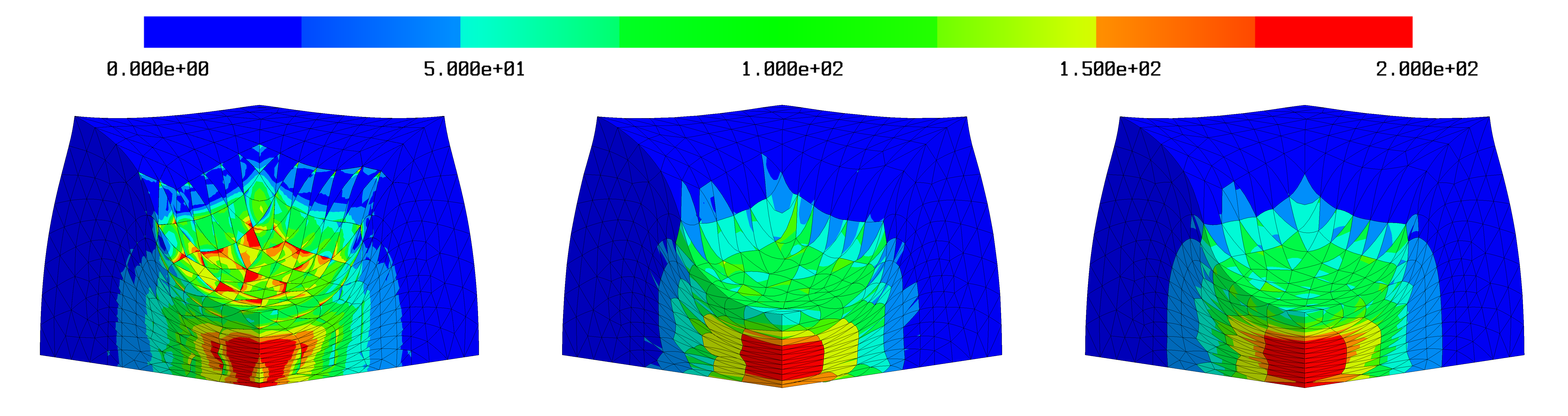}
   
    \caption{Norm of 1\textsuperscript{st}~Piola--Kirchhoff stress tensor on the deformed domain for block compressible problem with load $\vT=(0,0,-400)$.  Left: \emph{NDTNS}. Middle: \emph{CSMFEM}. Right: \emph{STP}.}
    \label{fig:comp3}
\end{figure}

\begin{figure}[ht!]
    \centering
    \includegraphics[width=0.8\linewidth]{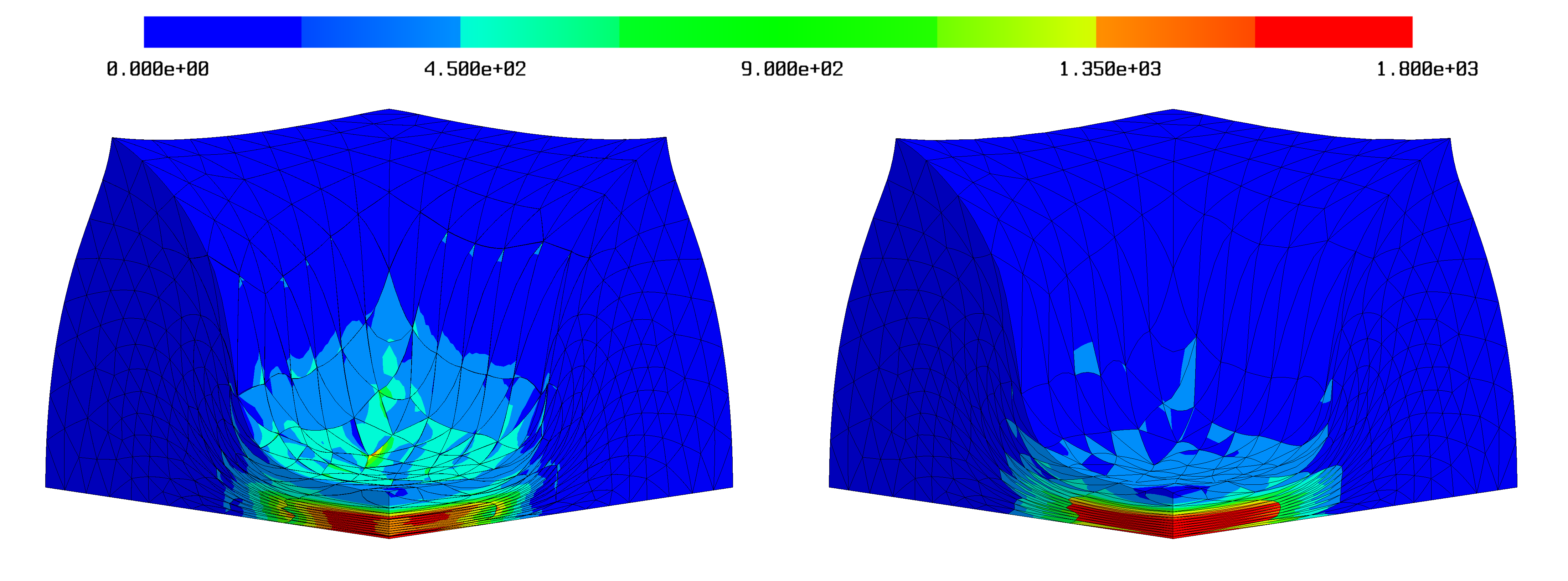}
    \caption{Norm of 1\textsuperscript{st}~Piola--Kirchhoff stress tensor on deformed domain for block compressible problem
    with load $\vT=(0,0,-1000)$. Left: \emph{NDTNS}. Right: \emph{STP}.}
    \label{fig:comp4}
\end{figure}

\section{Conclusion}
\label{sec:conclusion}
We have presented a novel mixed finite element method for solving incompressible large deformation hyperelasticity problems. By adding the deformation gradient as an additional strain unknown, we were able to extend the mass-conserving mixed stress (MCS) method, which has been successfully applied for pressure-robust and mass-conserving discretizations of the Stokes equations, to incompressible finite elasticity. The key point was the computation of a discrete Riesz representative of the arising distributional deformation gradient, the square-integrable auxiliary deformation gradient. Element-wise static condensation can be applied to solve a system that consists of only displacement-based degrees of freedom. We investigated a (quasi-)Newton-Raphson solving scheme with stabilization options. We presented a comprehensive comparison of the proposed method with the CSMFEM, a Taylor--Hood-like, and an STP-like method of nonlinear incompressible elasticity on several challenging numerical benchmark examples. The proposed method showed the best balance between convergence and stability among the four investigated methods.

A stability term well-known from hybrid discontinuous Galerkin (HDG) methods has been established to guarantee stability in regions of massive deformations, especially strong compression. Future research will investigate an optimal choice of the stabilization term and possible automatic adaptive stabilization procedures to guarantee the optimal (super-)convergence of the displacement and the stress fields.

\section*{Acknowledgements}
Guosheng Fu acknowledges support by the National Science Foundation (USA) Grant DMS-2410741.
Michael Neunteufel acknowledges support by the National Science Foundation (USA) Grant DMS-219958.

\bibliographystyle{acm}
\bibliography{cites.bib}
  
\end{document}